\documentclass{amsart}

\usepackage{enumitem}

\newcommand{\A}{\mathcal{A}}
\newcommand{\N}{\mathbb{N}}
\newcommand{\R}{\mathbb{R}}
\newcommand{\C}{\mathbb{C}}
\newcommand{\F}{\mathcal{F}}
\newcommand{\M}{\mathcal{M}}
\newcommand{\Proj}{\mathcal{P}}
\newcommand{\Cx}{\C \langle x_1,\ldots,x_d\rangle}

\newcommand{\dom}{\mathrm{dom}}
\newcommand{\im}{\mathrm{im}}
\newcommand{\sa}{\mathrm{sa}}
\newcommand{\1}{\mathbf{1}}
\newcommand{\norm}[1]{\left\Vert #1\right\Vert}
\DeclareMathOperator{\tr}{Tr}
\DeclareMathOperator{\rk}{rk}

\makeatletter
\def\moverlay{\mathpalette\mov@rlay}
\def\mov@rlay#1#2{\leavevmode\vtop{%
\baselineskip\z@skip \lineskiplimit-\maxdimen
\ialign{\hfil$#1##$\hfil\cr#2\crcr}}}
\makeatother

\def\lff{\moverlay{(\cr<}}
\def\rff{\moverlay{)\cr>}}

\newcommand{\FF}{{\C \lff x_1,\ldots ,x_d \rff}}

\theoremstyle{definition}
\newtheorem{dfn}{Definition}
\theoremstyle{plain}
\newtheorem{thm}[dfn]{Theorem}

\newtheorem{prop}[dfn]{Proposition}
\newtheorem{cor}[dfn]{Corollary}
\newtheorem{lm}[dfn]{Lemma}
\theoremstyle{remark}
\newtheorem{exam}[dfn]{Example}
\newtheorem{rem}[dfn]{Remark}
\title[Convergence of nc rational functions evaluated in random matrices]{Convergence for noncommutative rational functions evaluated in random matrices
}

\author{Beno\^\i{}t Collins}
\address{Department of Mathematics, Kyoto University, Kitashirakawa Oiwake-cho, Sakyo-ku, 606-8502, Japan}
\email{collins@math.kyoto-u.ac.jp}

\author{Tobias Mai}
\address{Saarland University, Department of Mathematics, D-66123 Saarbr\"ucken, Germany}
\email{mai@math.uni-sb.de}

\author{Akihiro Miyagawa}
\address{Department of Mathematics, Kyoto University, Kitashirakawa Oiwake-cho, Sakyo-ku, 606-8502, Japan}
\email{miyagawa.akihiro.43v@st.kyoto-u.ac.jp}

\author{F\'elix Parraud} 
\address{Universit\'e de Lyon, ENSL, UMPA, 46 all\'ee d'Italie, 69007 Lyon. \\ Department of Mathematics, Kyoto University, Kitashirakawa Oiwake-cho, Sakyo-ku, 606-8502, Japan}
\email{felix.parraud@ens-lyon.fr}

\author{Sheng Yin $^{1,2}$}
\address{$^{1}$IMT, Universit\'e de Toulouse, UPS, F-31400 Toulouse, France}
\address{$^{2}$Department of Mathematics, Baylor University, Waco, TX 76706, USA}
\email{sheng\_yin@baylor.edu}

\begin{document}
\maketitle
\begin{abstract}
    One of the main applications of free probability is to show that for appropriately chosen independent copies of $d$ random matrix models, any noncommutative polynomial in these $d$ variables has a spectral distribution that converges asymptotically and can be described with the help of free probability. 
    This paper aims to show that this can be extended to noncommutative rational functions, answering an open question by Roland Speicher.
    
    This paper also provides a noncommutative probability approach to approximating the free field. At the algebraic level, its construction relies on the approximation by generic matrices. On the other hand, it admits many embeddings in the algebra of operators affiliated with a $II_1$ factor. A consequence of our result is that, as soon as the generators admit a random matrix model, the approximation of any self-adjoint noncommutative rational function by generic matrices can be upgraded at the level of convergence in distribution.  
\end{abstract}

\section{Introduction}

Following the earlier work of Ching \cite{ching73}, Avitzour \cite{Avitzour82} and Voiculescu \cite{Voi85} introduced a reduced free product $C^*$-algebra. In particular, Voiculescu noticed that the underlying concept of freeness can be interpreted as some highly noncommutative analog of the notion of independence in classical probability theory; this insight motivated him to develop what became known as free probability. His early investigations lead him, in particular, to the free central limit theorem in the spirit of other noncommutative central limit theorems such as \cite{FQ83}, \cite{HWP80}. Generally, non-trivial free products of groups yield ICC class groups, and their von Neumann algebra is a factor. Therefore they have only one finite trace, and it was natural to study free product factors from the point of view of noncommutative probability spaces -- a pair consisting of an algebra and a trace. This point of view has been spectacularly successful, and arguably one of its most significant -- and initially unexpected -- achievements was to describe the limiting spectrum of noncommutative polynomials in i.i.d. random matrices. This phenomenon is known as asymptotic freeness.  
It was initially described in \cite{MR1094052,MR1601878}.
It has been subsequently enhanced in many directions, and a notable direction of improvement was the study of the norm of a random matrix (strong asymptotic freeness). A beautiful breakthrough was made in  2005 by  Haagerup, and Thorbj\o rnsen in \cite{HT}, where they proved the almost sure convergence of the norm of a polynomial evaluated in independent GUE matrices. This method was then refined in multiple directions; the case of Gaussian orthogonal and symplectic matrices was tackled in \cite{SCH}, that of Wigner matrices in \cite{CD,AND}, and that of Wigner and deterministic matrices in \cite{male,belin-capi}. Besides that, Collins and Male proved in \cite{MR3205602} the same result with unitary Haar matrices instead of GUE matrices by mapping this problem with  Male's previous results. Those result has allowed turning an area of pure mathematics into a very useful tool for applied mathematics that relies heavily on random matrix models. 

The models whose limiting behavior is well-understood thanks to asymptotic freeness involve the arithmetic operations of multiplication, addition, and scalar multiplication. Other useful operations have also been successfully studied through free probability, such as taking matrix values or Hadamard products (asymptotic freeness with amalgamation).
We refer to \cite{MR3585560} for an exposition of many 
classical and recent results in this direction.
Recently, it has also been possible to involve systematically
the composition with smooth functions, even at the level of 
strong convergence, cf. \cite{p1,p2}.

However, one arithmetic operation that remains largely unexplored in the context of random matrix models is the inverse. The purpose of this paper is to address this question.
A recent result in \cite{Yin18} shows that taking the inverse is a stable operation if the strong convergence replaces the convergence in distribution for matrices.
More precisely, for a sequence of matrices that strongly converges in distribution towards a limiting object -- an operator which we also call a (noncommutative) random variable -- the sequence of their inverses will eventually be well-defined and strongly converges in distribution to the inverse of the limiting random variable, provided that this random variable has a bounded inverse.
Furthermore, such a result can be extended from an inverse to noncommutative rational functions in multiple variables (a counterpart of commutative rational functions that has been developed by many pioneers, see, for examples \cite{Ami66,Coh06}) by a recursive structure of rational functions (and their representing rational expressions) or by a linearization trick for rational functions.

However, to go beyond the case of bounded evaluations, a problem that one faces about the inverse is that using it might fail to result in a well-defined model when the inverse is performed on non-invertible matrices. 
On the other hand, the limiting object has recently been at the center of the attention of free probabilists, and many breakthroughs have been obtained, see among others \cite{MSY19}.
Incidentally, the limiting theory relies on the theory of noncommutative rational functions and the embedding question of the rational functions in the generating operators into the algebra of unbounded operators affiliated with the underlying von Neumann algebra.
This embedding question was affirmatively answered long ago in \cite{Linnell1993}, whose goal was to answer the Atiyah conjecture for some families of groups, including the free groups.
It was recently noted that this result also answered the well-definedness question for rational functions in freely independent Haar unitary random variables in the context of free probability.
Moreover, in \cite{MSY19}, the well-definedness was further proved for a large family of random variables beyond the free Haar unitaries.

Let us also mention that there are many natural random matrix models involving the inverse operation and that this is an important topic nowadays, see e.g.
\cite{EKN20,MY21}.
One goal of our manuscript is to provide a unified approach to the study of the limiting
spectral distribution under such generality.
Therefore, the natural questions are:
\begin{itemize}
    \item can we make sense of random matrix models involving inverses?
    \item do they converge towards their natural limiting candidates, whose properties have been unveiled recently?
\end{itemize}

Speicher asked these questions during a meeting at MFO in 2019, \cite{MFO2019}.
Partial answers have been given under some assumptions, such as bounded evaluation and specific random matrix models, cf. \cite{Yin18, EKN20, Yin20}.
Let us mention that \cite[Theorem 5.2]{HS07} considers, for freely independent semicirculars $(x_1,x_2,x_3,x_4)$, the operator $(x_1+ix_2)(x_3+ix_4)^{-1}$, whose radial part fits in the context of our study.
While our results do not tell anything about the convergence of the spectral measure to the Brown measure, we can say that the radial part of $(X^N_1 + i X_2^N)(X^N_3 + i X_4^N)^{-1}$, for independent GUE matrices $(X_1^N,X_2^N,X_3^N,X_4^N)$,
converges to the radial part of $(x_1+ix_2)(x_3+ix_4)^{-1}$, which sheds some additional light on this operator. 
The purpose of this paper is to settle these questions in a very general setup.
Our main results can be stated as follows.

\begin{thm}
\label{thm:main}
	Let $X^N=(X_1^N,\dots,X_{d_1}^N)$ be a $d_1$-tuple of self-adjoint random matrices and let $U^N=(U_1^N,\dots,U_{d_2}^N)$ be a $d_2$-tuple of unitary random matrices. Further, let $R$ be a non-degenerate square matrix-valued noncommutative rational expression in $d=d_1+d_2$ variables which is self-adjoint of type $(d_1,d_2)$; see Definition \ref{def:rational_expression_selfadjoint}. Suppose that the following conditions are satisfied:
	\begin{enumerate}
	    \item\label{thm:main-cond1} $(X^N,U^N)$ converges almost surely in $\ast$-distribution towards a $d$-tuple of noncommutative random variables $(x,u)$ in some tracial $W^\ast$-probability space $(\M,\tau)$ satisfying the regularity condition $\Delta(x,u)=d$; see Sections \ref{subsec:unbounded} and \ref{subsec:Delta}.
	    \item\label{thm:main-cond2} For $N$ large enough $R(X^N,U^N)$ is well-defined almost surely.
	\end{enumerate}
	Then $R(x,u)$ is well-defined, and the empirical measure of $R(X^N,U^N)$ converges almost surely in law towards the analytic distribution of $R(x,u)$.

The assumption \ref{thm:main-cond2} is satisfied for random matrix models $(X^N,U^N)$ whose law on $M_N(\C)_\sa^{d_1} \times U_N(\C)^{d_2}$ is absolutely continuous with respect to the product measure of the Lebesgue measure on $M_N(\C)_\sa$ and the Haar measure on $U_N(\C)$.

In particular, the assumptions \ref{thm:main-cond1} and \ref{thm:main-cond2} are satisfied for random matrix models satisfying the following conditions:
\begin{itemize}
    \item $(X^N,U^N)$ are almost surely asymptotically free.
    \item The law of each $X_j^N$ has a density with respect to the Lebesgue measure on $M_N(\C)_\sa$ and its eigenvalue distribution almost surely converges weakly to some compactly supported probability measure on $\R$ that is non-atomic.
    \item $U^N$ are i.i.d. Haar distributed.
\end{itemize}
\end{thm}

The main part of Theorem \ref{thm:main} will be proved in Theorem \ref{thm:evaluation_maximalDelta} and Theorem \ref{conver}. Theorem \ref{thm:evaluation_maximalDelta} ensures that $r(x,u)$ is well-defined as long as $\Delta(x,u)=d$, whereas Theorem \ref{conver} proves that the convergence in $\ast$-distribution implies the convergence of the empirical measure. Note that we prove this theorem for any sequence of deterministic matrices which satisfies assumptions \ref{thm:main-cond1} and \ref{thm:main-cond2}. Thus this theorem can be used outside of the field of random matrices. Finally in Theorem \ref{thm:evaluation_rand_mat_abscont}, we prove that assumption \ref{thm:main-cond2} is satisfied for ``absolutely continuous'' random matrix models $(X^N,U^N)$. In combination with this, Corollary \ref{cor:maximal_Delta} ensures that the particular random matrix model $(X^N,U^N)$ satisfies the assumptions \ref{thm:main-cond1} and \ref{thm:main-cond2} of Theorem \ref{thm:main}. 

An interpretation of our results is as follows: the free field
$\C\lff x_1,\dots,x_d\rff$ together with a $*$-structure can be endowed
with a noncommutative probability structure
through its embedding in the $*$-algebra of operators affiliated to a $II_1$ factor.
Let us elaborate on this noncommutative probability structure.
To each selfadjoint rational function $R$ of the free
field one can associate a probability measure on $\R$ (i.e. an element of $\mathcal{P}(\mathbb{R})$),
which may have unbounded support. 
One sees that this map $\C\lff x_1,\dots,x_d\rff_{\sa}\to \mathcal{P}(\mathbb{R})$
 strictly generalizes the tracial map
on the $*$-algebra generated by the generators because, in general, elements of the free field do not have moments. 
In addition, it allows one to define a noncommutative probability structure directly on a $*$-free field without necessarily resorting to
von Neumann algebras and affiliated operators.  In this context, our result says that
any matrix approximation of the
free noncommutative tracial $*$-algebra generated by the generators of the free field
in the sense of noncommutative distribution convergence, can be upgraded into 
pointwise convergence of the map 
$\C\lff x_1,\dots,x_d\rff_{\sa}\to \mathcal{P}(\mathbb{R})$.
Consequently, this interpretation may allow us to read the information on the algebraic side (the free field) out of information on the probabilistic side (probability measures).
For example, the inner rank of a self-adjoint matrix $A$ over $\C\lff x_1,\dots,x_d\rff$ can be seen from the asymptotic proportion of zero eigenvalues in the spectrum of the evaluation of $A$ at any approximation matrices that model the free field.

For a $d$-tuple $X$ satisfying $\Delta(X)=d$, we know from \cite{MSY19} that the division closure $D(X)$ of $\C\langle X\rangle$ in the $\ast$-algebra $\widetilde{W^\ast(X)}$ of all closed and densely defined operators that are affiliated with $W^\ast(X)$ provides a model of the free field $\C\lff x\rff$. It would be interesting to find a criterion similar to \cite{DR97,Linnell2000} which allows us to decide whether an element in $\widetilde{W^\ast(X)}$ belongs to the division closure $D(X)$. If $d\geq 2$ and $X$ are free Haar unitaries, then $W^\ast(X)$ is isomorphic to the free group factor $L(\mathbb{F}_d)$; in this case, such a criterion was provided by Linnell in \cite{Linnell2000}, building on the paper \cite{DR97} by Duchamp and Reutenauer in which they proved a conjecture of Connes \cite{C94}.

This paper is organized as follows: following this introductory
section, section \ref{sec:preliminaries} gathers necessary
facts about noncommutative rational functions and expressions; section \ref{sec:evaluations} shows that the random matrix model is well-defined for a dimension large enough, and section \ref{sec:main} evaluates the limiting distribution.

Acknowledgments. The problem considered in this paper appeared in the context of 
the MSc studies of A. Miyagawa, under the supervision of B. Collins (cf \cite{Miyagawa-MSc}), and related questions were discussed during the visit of T. Mai in Kyoto in 2019. T. Mai is grateful for the great hospitality of B. Collins and the entire Department of Mathematics at Kyoto University. F. Parraud, T. Mai, and S. Yin benefited from the hospitality of MFO, during which R. Speicher stated the conjecture leading to this paper.
The authors thank G. C\'ebron, A. Connes, M. de la Salle, and R. Speicher for valuable discussions. We thank an anonymous referee for a careful reading of our manuscript and for valuable comments and suggestions.

BC was supported by JSPS KAKENHI 17K18734 and 17H04823.
FP was partially supported by Labex Milyon (ANR-10-LABX-0070) of Universit\'e de Lyon.
SY was supported by ANR project MESA.

\section{Preliminaries}\label{sec:preliminaries}

\subsection{Noncommutative rational functions and expressions}

Let us denote by $\Cx$ the algebra of noncommutative polynomials over $\C$ in the indeterminates $x_1,\ldots,x_d$.
It is well-known that for its commutative counterpart, namely the ring of commutative polynomials, one can uniquely construct the field of fractions of this polynomial ring by the quotients of polynomials.
However, constructing a skew field of fractions of $\Cx$ is highly non-trivial.
Moreover, there exist skew fields of fractions of $\Cx$ which are not isomorphic (see, for example, \cite[Exercise 7.2.13]{Coh06}).
Nevertheless, there exists a unique field of fractions of $\Cx$, which has some universal property.
We call this skew field of fractions of $\Cx$ the \emph{free field} and denote it by $\C\lff x_1,\dots,x_d\rff$.
An element in the free field is called a \emph{noncommutative rational function}.

Since the precise definition of the universal
property of the free field is not relevant to this paper, we refer the interested reader to
\cite[Chapter 7]{Coh06} for a more detailed description (as well as some ring-theoretic construction) of the free field.
We take the existence of the free field for granted and apply some recent results about it.

Like a commutative rational function can be represented by a class of quotients of polynomials, a noncommutative rational function can be represented by a class of noncommutative rational expressions.
One can think of a noncommutative rational expression as a representation of a noncommutative rational function.
Actually, in \cite{Ami66}, Amitsur constructed the free field $\C\lff x_1,\dots,x_d\rff$ from noncommutative rational expressions (see also \cite[Section 2]{K-VV12}).

More precisely, \emph{noncommutative rational expressions} are syntactically valid combinations of $\C$ and symbols $x_1,\ldots,x_d$ with $+$, $\cdot$, ${}^{-1}$, and $()$, which are respectively corresponding to addition, multiplication, taking inverse, and ordering these operations. For the sake of completeness, let us mention that polynomial expressions are obtained in precisely the same manner but without involving inverses.
We admit that this definition, though easy to grasp, is not entirely rigorous as it relies on the tacit agreement about what is meant by syntactically valid.
Thus we refer here to \cite{HW15} for an alternative definition based on the graph theory, by which arithmetic operations for rational expressions can be interpreted as operations on graphs (see also \cite[Section IV.1]{Yin20}).
We emphasize that noncommutative rational expressions (in contrast to noncommutative rational functions, which we will define later) are formal objects obeying no arithmetic rules like commutativity or associativity.
For example, though clearly, they represent the same function, $x_1 + (-1) \cdot x_1$ and $0$ are two distinct rational expressions.
Similarly, $x_1 \cdot (x_2 \cdot x_1)$ and $(x_1 \cdot x_2) \cdot x_1$ are different noncommutative rational expressions. Still, since they show the same behavior when evaluated on associative algebras, we will write shorthand $x_1 \cdot x_2 \cdot x_1$ or even $x_1 x_2 x_1$ for better legibility as the inherent ambiguity does not cause any problems.

One can also define \emph{matrix-valued noncommutative rational expressions}; see Definition 2.1 in \cite{K-VV09}.
Those are possible combinations of symbols $A \otimes 1$ and $A \otimes x_j$ for $j=1,\ldots,d$, for each rectangular matrix $A$ over $\C$ of arbitrary size, with $+$, $\cdot$, ${}^{-1}$, and $()$, where the operations are required to be compatible with the matrix sizes. Notice that $\otimes$ has only symbolic meaning here but will turn into the ordinary tensor product (over $\C$) under evaluation, as will be defined below.

Let us enumerate the rules which allow us to recursively compute for every matrix-valued noncommutative rational expression $R$ the domain $\dom_\A(R)$ of $R$ for every unital complex algebra $\A$ and evaluations $R(X)$ of $R$ at any point $X\in\dom_\A(R)$; note that the evaluation $R(X)$ of a $p \times q$ matrix-valued noncommutative rational expression $R$ and every point $X\in\dom_\A(R)$ belongs to $M_{p\times q}(\C) \otimes \A \cong M_{p\times q}(\A)$.
\begin{itemize}
     \item If $R = A \otimes 1$ for some $A\in M_{p \times q}(\C)$, then $R$ is a $p\times q$ matrix-valued noncommutative rational expression with $\dom_\A(R) := \A^d$ and $R(X) := A \otimes 1_\A$ for every $X \in \A^d$.
     \item If $R = A \otimes x_j$ for some $A\in M_{p \times q}(\C)$ and $1\leq j \leq d$, then $R$ is a $p\times q$ matrix-valued noncommutative rational expression with $\dom_\A(R) := \A^d$ and $R(X) := A \otimes X_j$ for every $X=(X_1,\dots,X_d) \in \A^d$.
     \item If $R_1,R_2$ are $p \times q$ matrix-valued noncommutative rational expressions, then $R_1 + R_2$ is a $p \times q$ matrix-valued noncommutative rational expression with $\dom_\A(R_1 + R_2) := \dom_\A(R_1) \cap \dom_\A(R_2)$ and $(R_1 + R_2)(X) := R_1(X) +_\A R_2(X)$ for every $X\in \dom_\A(R_1 + R_2)$, where $+_\A$ stands for the addition $M_{p \times q}(\A) \times M_{p\times q}(\A) \to M_{p \times q}(\A)$.
     \item If $R_1,R_2$ are $p\times q$ respectively $q\times r$ matrix-valued noncommutative rational expressions, then $R_1 \cdot R_2$ is a $p \times r$ matrix-valued noncommutative rational expression with $\dom_\A(R_1 \cdot R_2) := \dom_\A(R_1) \cap \dom_\A(R_2)$ and $(R_1 \cdot R_2)(X) := R_1(X) \cdot_\A R_2(X)$ for every $X\in \dom_\A(R_1 \cdot R_2)$, where $\cdot_\A$ stands for the matrix multiplication $M_{p\times q}(\A) \times M_{q\times r}(\A) \to M_{p \times r}(\A)$.
     \item If $R$ is a $p \times p$ matrix-valued noncommutative rational expression, then $$\dom_\A(R^{-1}) := \{X\in \dom_\A(R) \mid \text{$R(X)$ is invertible in $M_p(\A)$}\}$$ and $R^{-1}(X) := R(X)^{-1}$ for every $X \in \dom_\A(R^{-1})$.
 \end{itemize}
 
 Note that the (scalar-valued) noncommutative rational expressions we introduced before belong to the larger class of $1 \times 1$ matrix-valued noncommutative rational expressions; see Remark 2.11 in \cite{K-VV09}.
 
 For the reader's convenience, we introduce two types of matrix-valued noncommutative rational expressions, which are important in a practical sense. 
\begin{itemize}
 \item A noncommutative rational expression evaluated in formal tensor products of matrices and formal variables like as 
    $$R = r(A_1 \otimes x_1, A_2 \otimes x_2, \ldots, A_d \otimes x_d)$$ 
    where $r$ is a (scalar-valued) noncommutative rational expression and $A_i \in M_p(\C)$ for $1 \le i \le d$. In other words, in this case, we amplify formal variables by matrices and then consider their (scalar-valued) rational expression. 
    \item A matrix which consists of (scalar-valued) noncommutative rational expressions 
    $$R = (r_{i j})_{1 \le i \le p, 1 \le j \le q}.$$
    This can be seen as a $p \times q$ matrix-valued noncommutative rational expression
    by identifying with $\sum_{i j}(a_{i} \otimes 1) r_{i j} (b_{j} \otimes 1)$ where $a_i \in M_{p \times 1}(\C)$ and $b_j \in M_{1 \times q}(\C)$ are standard basis of $\C^p$ and $\C^q$.
    We will implicitly use this viewpoint later (for example, in the proof of Proposition \ref{prop:full_sa_evaluation}). 
\end{itemize}

A class of matrix-valued noncommutative polynomial expressions is affine linear pencils. An \emph{affine linear pencil (in $d$ variables with coefficients from $M_k(\C)$)} is a $k\times k$ matrix-valued noncommutative polynomial expression of the form
$$A=A_0 \otimes 1 + A_1 \otimes x_1 + \dots + A_d \otimes x_d$$
with coefficient matrices $A_0,A_1,\dots,A_d$ belonging to $M_k(\C)$.
Notice, once again, that we omit the parentheses for better readability, as each syntactically valid placement of parentheses will produce the same result under evaluation.
If $\A$ is any unital complex algebra and $X\in \A^d$, then
$$A(X) = A_0 \otimes 1_\A + A_1 \otimes X_1 + \dots + A_d \otimes X_d \in M_k(\C) \otimes \A \cong M_k(\A).$$
 
 Of particular interest are matrix evaluations. For each matrix-valued noncommutative rational expression $R$, we put
 $$\dom_{M(\C)}(R) := \coprod_{N=1}^\infty \dom_{M_N(\C)}(R),$$
 i.e., $\dom_{M(\C)}(R)$ is the subset of all square matrices over $\C$ where evaluation of $R$ is well-defined.
 A matrix-valued noncommutative rational expression $R$ is said to be \emph{non-degenerate} if it satisfies $\dom_{M(\C)}(R) \neq \emptyset$. In the sequel, we will make use of the following important fact.
 
 \begin{thm}[Remark 2.3 in \cite{K-VV09}]\label{thm:matrix_domain_large_dimension}
 Let $R$ be a non-degenerate matrix-valued noncommutative rational expression. Then there exists some $N_0=N_0(R)\in\N$ such that $\dom_{M_N(\C)}(R) \neq \emptyset$ for all $N\geq N_0$.
 \end{thm}
 
 Two non-degenerate matrix-valued noncommutative rational expressions $R_1, R_2$ are called \emph{$M(\C)$-evaluation equivalent} if the condition $R_1(X) = R_2(X)$ is satisfied for all $X\in \dom_{M(\C)}(R_1) \cap \dom_{M(\C)}(R_2)$.
 
 As we mentioned earlier, one can construct the free field out of noncommutative rational expressions.
 This construction can be done by evaluating (scalar-valued) noncommutative rational expressions on scalar-valued matrices.
 For a non-degenerate noncommutative rational expression $r$, we denote by $[r]$ its equivalence class of noncommutative rational expressions with respect to $M(\C)$-evaluation equivalence. We endow the set of all such equivalence classes with the arithmetic operations $+$ and $\cdot$ defined by $[r_1] + [r_2] := [r_1 + r_2]$ and $[r_1] \cdot [r_2] := [r_1 \cdot r_2]$. Notice that the arithmetic operations are indeed well-defined as one has $\dom_{M(\C)}(r_1) \cap \dom_{M(\C)}(r_2) \neq \emptyset$ for any two non-degenerate scalar-valued noncommutative rational expressions $r_1$ and $r_2$; see the footnote on page 52 of \cite{K-VV12}, for instance.
 It is known (see Proposition 2.2 in \cite{K-VV12}) that the set of all equivalence classes of noncommutative rational expressions with respect to $M(\C)$-evaluation equivalence endowed with the arithmetic operations $+$ and $\cdot$ forms the free field $\FF$.

\subsection{Linearization}
 
Let us recall the following terminology introduced in \cite[Definition 4.10]{HMS18}.

\begin{dfn}[Formal linear representation]\label{def:formal_linear_representation}
Let $R$ be a $p \times q$ matrix-valued noncommutative rational expression in the variables $x_1,\dots,x_d$. A \emph{formal linear representation $\rho=(u,A,v)$ of $R$ (of dimension $k$)} consists of an affine linear pencil
$$A=A_0 \otimes 1 + A_1 \otimes x_1 + \dots + A_d \otimes x_d$$
in $d$ variables and with coefficients $A_0,A_1,\dots,A_d$ from $M_k(\C)$ and matrices $u \in M_{p \times k}(\C)$ and $v\in M_{k \times q}(\C)$, such that the following condition is satisfied: for every unital complex algebra $\A$, we have that $\dom_\A(R) \subseteq \dom_\A(A^{-1})$ and for each $X \in \dom_\A(R)$ it holds that $R(X) = u A(X)^{-1} v$, where $A(X) \in M_k(\A)$.
\end{dfn}

Note that we use here a different sign convention by requiring $R(X) = u A(X)^{-1} v$ instead of $R(X) = - u A(X)^{-1} v$; this, however, does not affect the validity of the particular results that we will take from \cite{HMS18}. Furthermore, as we will exclusively work with formal linear representations for matrix-valued noncommutative rational expressions, we will go without specifying them as matrix-valued formal linear representations like it was done in \cite{HMS18}.

It follows from \cite[Theorem 4.12]{HMS18} that, indeed, every matrix-valued noncommutative rational expression $R$ admits a formal linear representation $\rho=(u,A,v)$. For the reader's convenience, we include with Theorem \ref{linear} the precise statement and its constructive proof.
In doing so, we will see that Algorithm 4.11 in \cite{HMS18}, on which the proof of Theorem 4.12 in the same paper relies, provides a formal linear representation $\rho=(u,A,v)$ of the $p \times q$ matrix-valued noncommutative rational expression $R$ with the additional property that the dimension $k$ of $\rho$ is larger than $\max\{p,q\}$ and that both $u$ and $v$ have maximal rank; we will call such $\rho$ \emph{proper}.
Note that if $R$ is a scalar-valued rational expression, then a proper formal linear representation $\rho$ simply means that $u$ and $v$ are non-zero vectors.
In general, due to the restriction $k\geq \max\{p,q\}$, we have that the rank of $u$ is $p$ and the rank of $v$ is $q$ for any proper formal linear representation $\rho=(u,A,v)$ of $R$.
This notion of proper formal linear representation will be important in the sequel.

\begin{thm}[Theorem 4.12 in \cite{HMS18}]\label{linear}
Every matrix-valued noncommutative rational expression admits a formal linear representation in the sense of Definition \ref{def:formal_linear_representation} which is also proper.
\end{thm}

\begin{proof}
Here, we give the algorithm which inductively builds a proper linear representation of any matrix-valued noncommutative rational expression.
For $R = A \otimes 1$ or $R = A \otimes x_j$ for some $A \in M_{p \times q}(\C)$ and $1 \le j \le d$ we have
$$ R(X) = 
\left(\begin{array}{cc}
I_p & 0_{p \times q}
\end{array}\right)
\left(\begin{array}{cc}
I_p \otimes 1_{\A} & -R(X) \\
0_{q \times p} & I_q \otimes 1_{\A}
\end{array}\right)^{-1}
\left(\begin{array}{c}
0_{p \times q} \\
I_q
\end{array}\right)
$$
where $I_p \in M_p(\C)$ is an identity matrix. We obtain a proper formal linear representation in this way.

If the $p \times q$ matrix-valued noncommutative expressions $R_1$ and $R_2$ admit proper formal linear representations $(u_1,A_1,v_1)$ and $(u_2,A_2, v_2)$
then we have
$$
(R_1 + R_2)(X) = 
\left(\begin{array}{c c}
u_1 & u_2
\end{array}\right)
\left(\begin{array}{c|c}
A_1(X) & 0_{k_1 \times k_2} \\ \hline
0_{k_2 \times k_1} & A_2(X)
\end{array}\right)^{-1}
\left(\begin{array}{c}
v_1 \\
v_2
\end{array}\right).
$$
This gives us a proper formal linear representation since $(u_1\ u_2)$, resp. $(v_1\ v_2)^T$ is of rank $p$, resp. $q$.

If $R_1, R_2$ are $p \times q$, resp. $q \times r $ matrix-valued noncommutative rational expressions and admit formal linear representations $(u_1, A_1, v_1)$, resp. $(u_2, A_2, v_2)$ of dimension $k_1$, resp. $k_2$  then we have  
$$
 (R_1 \cdot R_2)(X) =
\left(\begin{array}{c c}
u_1 & 0_{p \times k_2}
\end{array}\right)
\left(\begin{array}{c|c}
A_1(X) & -v_1 u_2 \\ \hline
0_{k_2 \times k_1} & A_2(X)
\end{array}\right)^{-1}
\left(\begin{array}{c}
0_{k_1 \times r} \\
v_2
\end{array}\right).
$$
We obtain a proper formal linear representation since $(u_1\ 0_{p\times k_2})$, resp. $(0_{k_1\times r}\ v_2)^T$ is of rank $p$, resp. $q$.

If $R$ is a $p \times p$ matrix-valued noncommutative rational expression which admits a formal linear representation $(u,A,v)$ of dimension $k$, then we have
$$
R^{-1}(X) =
\left(\begin{array}{c c}
I_p & 0_{p \times k}
\end{array}\right)
\left(\begin{array}{c|c}
 0_{p\times p} & u \\ \hline
 v & A(X)
\end{array}\right)^{-1}
\left(\begin{array}{c}
-I_p \\
0_{k \times p}
\end{array}\right),
$$
where $X$ belongs to an appropriate domain for each step.
This gives us a proper formal linear representation.

Finally, since all matrix-valued noncommutative rational expressions can be represented by finitely many of the above steps, any matrix-valued noncommutative rational expression has such a formal linear representation that is proper.
\end{proof}

In the non-degenerate case, formal linear representations are connected with the concept of representations for noncommutative rational functions, which is used, for instance, in \cite{CR94,CR99}; this will be addressed in Remark \ref{rem:linear-full} and Remark \ref{rem:expressions_vs_functions}. Before, we need to recall the following terminology.

\begin{dfn}[Inner rank and fullness]
Let $\mathcal{R}$ be a ring. For $A \in M_{n \times m}(\mathcal{R})$ we define the \emph{inner rank} $\rho_{\mathcal{R}}(A)$ by
$$\rho_{\mathcal{R}}(A) = \min \{ r \in \N \mid A = BC, \ B \in M_{n \times r}(\mathcal{R}), \ C \in M_{r \times m}(\mathcal{R})\},$$
and $\rho_{\mathcal{R}}(0) = 0$. In addition we call $A$ \emph{full} if $\rho_{\mathcal{R}}(A) = \min\{n,m\}$.
\end{dfn}

\begin{rem}\label{rem:linear-full}
\begin{enumerate}[leftmargin=*]
\item Let $A$ be a matrix over noncommutative polynomials in a tuple $x=(x_1,\dots,x_d)$ of formal variables. According to Theorem 7.5.13 in \cite{Coh06} (see also A.2 in \cite{MSY18}), we have
$$\rho_{\C\langle x \rangle}(A)=\rho_{\C\lff x\rff}(A).$$
For this reason, we say $A$ is full for a square matrix $A$ over the noncommutative polynomials without mentioning which algebra we consider.  
\item\label{it:invertible_full} Let $A$ be an affine linear pencil in $x$ with coefficients taken from $M_k(\C)$. We may view $A$ as an element in $M_k(\C) \otimes \C\langle x \rangle \cong M_k(\C\langle x\rangle)$, i.e., $A=A(x)$ is considered as a matrix over the ring $\C\langle x\rangle$. We notice that if there exists a tuple $X\in M_N(\C)^d$ such that $A(X)$ is invertible in $M_k(\C) \otimes M_N(\C) \cong M_{kN}(\C)$, or equivalently, if $\dom_{M(\C)}(A^{-1}) \neq \emptyset$, then $A$ must be full. In fact, if $A$ is not full, then any factorization $A=BC$ with $B \in M_{k \times r}(\C\langle x \rangle)$ and $C \in M_{r \times k}(\C\langle x\rangle)$ for $r = \rho(A) < k$ yields under evaluation $A(X) = B(X) C(X)$ at any point $X\in M_N(\C)^d$, so that $A(X)$ is never invertible.

On the other hand, if $A$ is full, then $A$ is invertible as a matrix over $\C\lff x\rff$. Indeed, fullness and invertibility are equivalent for any skew field (see Lemma 5.20 in \cite{MSY18}).

\item\label{it:representation_full} Now, let $R$ be a non-degenerate matrix-valued noncommutative rational expression. From Theorem \ref{linear}, we know that there exists a formal linear representation $\rho=(u,A,v)$; in particular, we have that
\begin{align*}
\dom_{M(\C)}(R) & \subseteq \dom_{M(\C)}(A^{-1})\\ 
&=\coprod^\infty_{N=1} \{ X \in M_N(\C)^d \mid \text{$A(X)$ invertible in $M_{kN}(\C)$}\}.
\end{align*}
Since $R$ is non-degenerate, we find $X\in \dom_{M(\C)}(R)$; from the aforementioned inclusion and \ref{it:invertible_full}, we infer that $A$ is a full matrix.

\item\label{it:display_full} Suppose that $R$ is a non-degenerate $p\times p$ matrix-valued noncommutative rational expression such that $R^{-1}$ is non-degenerate as well. Let $\rho=(u,A,v)$ be a formal linear representation of $R$; we associate to $\rho$ the affine linear pencil
$$\tilde{A} := \begin{pmatrix} 0_{p \times p} & u\\ v & A \end{pmatrix}.$$
We claim that both $A$ and $\tilde{A}$ are full. For $A$, we already know from \ref{it:representation_full} that this is true. To check fullness of $\tilde{A}$, we use that $R^{-1}$ is non-degenerate, which guarantees the existence of some $X\in \dom_{M(\C)}(R^{-1})$. Since in particular $X\in \dom_{M(\C)}(R)$, we get as $\rho$ is a formal linear representation of $R$ that $A(X)$ is invertible and $R(X) = u A(X)^{-1} v$. Because $X\in \dom_{M(\C)}(R^{-1})$, we know that $R(X)$ is invertible. Hence, the Schur complement formula implies that the matrix $\tilde{A}(X)$ is invertible. Thanks to \ref{it:invertible_full}, this implies that the affine linear pencil $\tilde{A}$ is full.  
\end{enumerate}
\end{rem}

The following lemma explains that non-degenerate matrix-valued noncommutative rational expressions induce matrices over the free field in some very natural way.

\begin{lm}\label{lem:representing_rational_expression}
Let $R$ be a $p\times q$ matrix-valued noncommutative rational expression in $d$ formal variables. If $R$ is non-degenerate, then $x=(x_1,\dots,x_d) \in \dom_\FF(R)$ and consequently, $R(x) \in M_{p\times q}(\C\lff x\rff)$.
\end{lm}

\begin{proof}
Let us denote by $\mathfrak{R}_0$ the set of all non-degenerate matrix-valued noncommutative rational expressions $R$ which have the property $x \in \dom_{\C\lff x\rff}(R)$. We want to show that $\mathfrak{R}_0$ consists of all non-degenerate matrix-valued noncommutative rational expressions.
To verify this assertion, we proceed as follows. Firstly, we notice that both $R_1 + R_2$ and $R_1 \cdot R_2$ belong to $\mathfrak{R}_0$ whenever we take $R_1,R_2\in\mathfrak{R}_0$ for which the respective arithmetic operation is defined. Secondly, we consider some $R\in \mathfrak{R}_0$ which is of size $p \times p$ and has the property that $R^{-1}$ is non-degenerate. By Theorem \ref{linear}, there exists a formal linear representation $\rho=(u,A,v)$ of $R$, say of dimension $k$, and according to Remark \ref{rem:linear-full} \ref{it:display_full} we have that both $A$ and the associated affine linear pencil 
$$\tilde{A} := \begin{pmatrix} 0_{p \times p} & u\\ v & A \end{pmatrix}$$
are full, i.e., $A(x) \in M_k(\C\langle x\rangle)$ and $\tilde{A}(x) \in M_{k+p}(\C\langle x\rangle)$ become invertible as matrices over the free field $\C\lff x\rff$. Since $x\in\dom_{\C\lff x\rff}(R)$ as $R \in \mathfrak{R}_0$, we get $R(x) = u A(x)^{-1} v$, because $\rho$ is a formal linear representation of $R$. Putting these observations together, the Schur complement formula yields that $R(x) \in M_p(\C\lff x\rff)$ must be invertible, i.e., $x\in\dom_{\C\lff x\rff}(R^{-1})$ and thus $R^{-1} \in \mathfrak{R}_0$, as desired.
\end{proof}

\begin{rem}\label{rem:representing_rational_expression}
With arguments similar to the proof of Lemma \ref{lem:representing_rational_expression} as based on Remark \ref{rem:linear-full} \ref{it:display_full}, one finds that if $R_1,R_2$ are non-degenerate matrix-valued noncommutative rational expressions satisfying $R_1(x) = R_2(x)$, then $R_1 \sim_{M(\C)} R_2$.
In other words, matrix identities over $\C\lff x\rff$ are preserved under well-defined matrix evaluations.
\end{rem}

\begin{rem}\label{rem:expressions_vs_functions}
In the scalar-valued case, the conclusion of Lemma \ref{lem:representing_rational_expression} can be strengthened slightly. For that purpose, it is helpful to denote the formal variables out of which the noncommutative rational expressions are built by $\chi_1,\dots,\chi_d$ to distinguish them from the variables $x_1,\dots,x_d$ of the free skew field $\FF$; note that accordingly $x_j = [\chi_j]$ for $j=1,\dots,d$. Now, if $r$ is any non-degenerate scalar-valued noncommutative rational expression in the formal variables $\chi_1,\dots,\chi_d$, then Lemma \ref{lem:representing_rational_expression} tells us that $x\in \dom_{\C\lff x\rff}(r)$ and $r(x) \in \C\lff x \rff$. Moreover, we have the equality $r(x) = [r]$. This can be shown with a recursive argument similar to the proof of Lemma \ref{lem:representing_rational_expression}; notice that if a non-degenerate scalar-valued noncommutative rational expression $r$ satisfies $r(x)=[r]$ and has the additional property that $r^{-1}$ is non-degenerate, then $r(x) = [r]$ is invertible in $\C\lff x\rff$, which implies $x \in \dom_{\C\lff x\rff}(r^{-1})$ with $r^{-1}(x) = [r]^{-1} = [r^{-1}]$.

This has the consequence that every formal linear representation $\rho=(u,A,v)$ of $r$ satisfies $[r] = r(x) = u A(x)^{-1} v$. In the language of \cite{CR94,CR99}, this means that the formal linear representation $\rho$ of $r$ induces a (pure and linear) representation of the corresponding noncommutative rational function $[r]$.
\end{rem}

\subsection{Self-adjointness for matrix-valued noncommutative rational expressions}

When evaluations of matrix-valued noncommutative rational expressions $R$ at points $X=(X_1,\dots,X_d) \in\dom_\A(R)$ for $\ast$-algebras $\A$ are considered, it is natural to ask for conditions which guarantee that the result $R(X)$ is self-adjoint, i.e., $R(X)^\ast = R(X)$.
We aim at formulating conditions which concern the matrix-valued noncommutative rational expression $R$ and which depend only on the ``type'' of its arguments $X_1,\dots,X_d$ but not on their concrete realization in some $\ast$-algebra $\A$.
The case when $X_1,\dots,X_d$ are all self-adjoint was discussed in \cite[Section 2.5.7]{HMS18}. The following definition generalizes the latter to matrix-valued noncommutative rational expressions in self-adjoint and unitary variables.

Recall that an element $X$ in a complex $\ast$-algebra $\A$ with unit $1_A$ is called \emph{self-adjoint} if $X^\ast = X$, and $U\in\A$ is said to be \emph{unitary} if $U^\ast U = 1_\A = U U^\ast$.

\begin{dfn}[Self-adjoint matrix-valued noncommutative rational expressions]
\label{def:rational_expression_selfadjoint}
Let $R$ be a square matrix-valued noncommutative rational expression in $d_1+d_2$ formal variables which we denote by $x_1,\dots,x_{d_1}$, $u_1,\dots,u_{d_2}$.
We say that $R$ is \emph{self-adjoint of type $(d_1,d_2)$}, if for every unital complex $\ast$-algebra $\A$ and all tuples $X=(X_1,\dots,X_{d_1})$ and $U=(U_1,\dots,U_{d_2})$ of self-adjoint respectively unitary elements in $\A$, the following implication holds:
$$(X,U) \in \dom_{\A}(R) \qquad \Longrightarrow \qquad R(X,U)^\ast = R(X,U)$$
\end{dfn}

One comment on this definition is in order. The reader might wonder why the matrix-valued noncommutative rational expressions do not explicitly involve other variables $u_1^\ast,\dots,u_{d_2}^\ast$ serving as a placeholder for the adjoints of $u_1,\dots,u_{d_2}$. In fact, for (scalar-valued) noncommutative rational expressions, such an approach was presented, for instance, in the appendix of \cite{EKN20} (a version for noncommutative polynomials also appears in \cite{Var18}); more precisely, noncommutative rational expressions in collections of self-adjoint variables $x$, non-self-adjoint variables $y$, and their adjoints $y^\ast$ were considered. For our purpose, however, this has the slight disadvantage that non-degenerate noncommutative rational expressions of this kind (take $r(y,y^\ast) = (y y^\ast - 1)^{-1}$, for example) may have no unitary elements in their domain. On the other hand, there are noncommutative rational expressions (or even noncommutative polynomial expressions such as $y y y^\ast + y^\ast y y^\ast$) which are not self-adjoint on their entire domain but self-adjoint on unitaries.

The following example illustrates Definition \ref{def:rational_expression_selfadjoint} and highlights the effect of having two types of variables.

\begin{exam}\label{ex:rational_expressions}
$x_1 + x_2^{-1}$, $i(u_1 - u_1^{-1})$ and $u_1^{-1}x_1^{-1}u_1$ are self-adjoint noncommutative rational expressions since we have for self-adjoint elements $X_1,X_2$ and a unitary $U_1$ in their domain, 
\begin{eqnarray*}
(X_1 + X_2^{-1})^* &=& X_1^* + (X_2^*)^{-1} = X_1 + X_2^{-1} \\
 (i(U_1 - U_1^{-1}))^* &=& -i(U_1^* - (U_1^*)^{-1}) = i(U_1 - U_1^{-1}) \\
(U_1^{-1}X_1^{-1}U_1)^* &=& U_1^* (X_1^*)^{-1} (U_1^*)^{-1} =  U_1^{-1}X_1^{-1}U_1.
\end{eqnarray*}
 However, $u_1 + u_2^{-1}$, $i(x_1 - x_1^{-1})$ and $x_1^{-1}u_1^{-1}x_1$ are not self-adjoint in our definition. So we need to be careful about the roles of formal variables when we consider self-adjoint rational expressions. 
For the matrix-valued case, the $2 \times 2$ respectively $1 \times 1$ matrix-valued noncommutative rational expressions
$$
\begin{pmatrix}
x_1^{-1} & u_1 \\
u_1^{-1} & x_2^{-1}
\end{pmatrix}
\qquad\text{and}\qquad
\begin{pmatrix}
u_1 & x_1 + iu_2
\end{pmatrix}
\begin{pmatrix}
1 & -i u_1\\ i u_1^{-1} & x_2
\end{pmatrix}^{-1}
\begin{pmatrix}
u_1^{-1} \\ x_1 - iu_2^{-1}
\end{pmatrix}
$$
are self-adjoint of type $(2,1)$ and $(2,2)$, respectively. 
\end{exam}
 
Like in \cite[Definition 4.13]{HMS18} for the case of self-adjoint arguments, we can introduce self-adjoint formal linear representations; see also \cite[Definition A.5]{EKN20} for the scalar-valued case.

Note that to make the machinery of self-adjoint linearizations ready for further applications, we will switch from now on to a more general situation.

\begin{dfn}[Self-adjoint formal linear representation] \label{def:formal_linear_representation_selfadjoint}
Let $R$ be a $p \times p$ matrix-valued noncommutative rational expression in $d$ formal variables $x_1,\dots,x_d$. A tuple $\rho=(Q,w)$ consisting of an affine linear pencil
$$Q = A_0 \otimes 1 + \sum^d_{j=1} \big(B_j \otimes x_j + B_j^\ast \otimes x_j^\ast\big)$$
in the formal variables $x_1,\dots,x_d$ and $x_1^\ast,\dots,x_d^\ast$, with coefficients being (not necessarily self-adjoint) matrices $B_1,\dots,B_d$ in $M_{k}(\C)$ for some $k\in\N$, some self-adjoint matrix $A_0\in M_k(\C)$ and some matrix $w \in M_{k \times p}(\C)$ is called a \emph{self-adjoint formal linear representation of $R$ (of dimension $k$)} if the following condition is satisfied: for every unital complex $\ast$-algebra $\A$ and all tuples $X=(X_1,\dots,X_d)$ of (not necessarily self-adjoint) elements in $\A$, one has
$$X \in \dom_\A(R) \qquad \Longrightarrow \qquad (X,X^\ast) \in \dom_\A(Q^{-1})$$
and for every $X \in \dom_\A(R)$ for which $R(X)$ is self-adjoint, it holds true that
$$R(X) = w^\ast Q(X,X^\ast)^{-1} w.$$
\end{dfn}

We point out that in contrast to the related concept introduced in \cite[Definition 4.13]{HMS18} the existence of a self-adjoint formal linear representation in the sense of the previous Definition \ref{def:formal_linear_representation_selfadjoint} does not enforce $R$ to be self-adjoint at any distinguished points in its domain. In fact, we have the following theorem: every square matrix-valued noncommutative rational expression admits a self-adjoint formal linear representation; this is analogous to \cite[Theorem 4.14]{HMS18}.

Like for formal linear representations, we will say that a self-adjoint formal linear representation $\rho=(Q,w)$ of a self-adjoint $p \times p$ matrix-valued noncommutative rational expression $R$ is \emph{proper} if the dimension $k$ of $\rho$ is larger than $p$ and if $w$ has a full rank (i.e., the rank of $w$ is $p$).

\begin{thm}\label{sa-lin}
Every square matrix-valued noncommutative rational expression in $d$ formal variables admits a self-adjoint formal linear representation in the sense of Definition \ref{def:formal_linear_representation_selfadjoint} which is proper.
\end{thm}

\begin{proof}
Let $\rho=(v,Q,w)$ be a formal linear representation of $R$ in the variables $x_1,\dots,x_{d}$ with the affine linear pencil $Q$ being of the form
$$Q = A_0 \otimes 1 + \sum^{d}_{j=1} B_j \otimes x_j.$$
We consider $\widetilde{\rho}=(\widetilde{Q},\widetilde{w})$ with the affine linear pencil
$$\widetilde{Q} = \widetilde{A}_0 \otimes 1 +  \sum^{d}_{j=1} \big(\widetilde{B}_j \otimes x_j + \widetilde{B}_j^\ast \otimes x_j^\ast\big)$$
in the variables $x_1,\dots,x_{d}$, $x_1^\ast,\dots,x_{d}^\ast$ given by
$$\widetilde{A}_0 := \begin{pmatrix} 0 & A_0^\ast\\ A_0 & 0 \end{pmatrix},\qquad \widetilde{B}_j := \begin{pmatrix} 0 & 0\\ B_j & 0 \end{pmatrix},\qquad \text{and}\qquad  \widetilde{w} := \begin{pmatrix} \frac{1}{2} v^\ast\\ w\end{pmatrix}.$$
One verifies that $\widetilde{\rho}=(\widetilde{Q},\widetilde{w})$ is a self-adjoint formal linear representation of $R$ which is moreover proper whenever $\rho$ is proper.
\end{proof}

Notice that if $R$ is a $p \times p$ matrix-valued noncommutative rational expression in $d_1+d_2$ formal variables $x_1,\dots,x_{d_1}$, $u_1,\dots,u_{d_2}$ which is self-adjoint of type $(d_1,d_2)$, then each self-adjoint formal linear representation of $R$ can be brought into the simplified form $\rho=(Q,w)$ with an affine linear pencil
$$Q = A_0 \otimes 1 + \sum^{d_1}_{j=1} A_j \otimes x_j + \sum^{d_2}_{j=1} \big(B_j \otimes u_j + B_j^\ast \otimes u_j^\ast\big)$$
in the formal variables $x_1,\dots,x_{d_1}$, $u_1,\dots,u_{d_2}$, $u_1^\ast,\dots,u_{d_2}^\ast$ with coefficients being self-adjoint matrices $A_0,A_1,\dots,A_{d_1}$ and (not necessarily self-adjoint) matrices $B_1,\dots,B_{d_2}$ in $M_{k}(\C)$ for some $k\in\N$ and some matrix $w \in M_{k \times p}(\C)$; indeed Theorem \ref{sa-lin} yields a self-adjoint formal linear representation of $R$ with an affine linear pencil in the formal variables $x_1,\ldots,x_{d_1}$, $x_1^{\ast},\ldots,x_{d_1}^*$ and $u_1,\ldots,u_{d_2}$, $u_1^{\ast},\ldots,u_{d_2}^{\ast}$, from which we obtain $Q$ of the asserted form by replacing $x_1^{\ast},\ldots,x_{d_1}^{\ast}$ by $x_1,\ldots,x_{d_1}$ and merging their coefficients. In particular, we have
$$(X,U) \in \dom_{\A}(R) \qquad \Longrightarrow \qquad (X,U,U^\ast) \in \dom_\A(Q^{-1})$$
and for every $(X,U) \in \dom_\A(R)$ it holds true that
$$R(X,U) = w^\ast Q(X,U,U^\ast)^{-1} w.$$

\begin{exam}
We return to the self-adjoint noncommutative rational expressions presented in Example \ref{ex:rational_expressions}. Let us construct a self-adjoint formal linearization of $x_1 + x_2^{-1}$. Using the algorithm from \cite{HMS18}, which we recalled in the proof of Theorem \ref{linear}, we obtain first a formal linear representation
$$ 
X_1 + X_2^{-1} =
\left(
\begin{array}{ccc}
 1 & 0 & 1 
 \end{array}
\right)
\left(
\begin{array}{ccc}
1 & -X_1 & 0  \\
0 & 1 & 0 \\
0 & 0 & X_2
\end{array} 
\right)^{-1}
\left(
\begin{array}{c}
0 \\
1\\
1
\end{array}\right).
$$
Out of the latter, we construct with the help of Theorem \ref{sa-lin} the self-adjoint formal linear representation
$$
X_1 + X_2^{-1} =
\left(
\begin{array}{cccccc}
 \frac{1}{2}&0&\frac{1}{2} &0 &1 &1 
\end{array}
\right)
\left(
\begin{array}{cccccc}
 0&0 &0 &1 & 0&0 \\
 0&0 &0 &-X_1 &1&0 \\
 0&0 &0 &0 &0 & X_2 \\
 1 & -X_1 & 0 & 0 & 0& 0 \\
 0 & 1 & 0 &0& 0& 0 \\
 0 & 0& X_2 & 0& 0& 0 
\end{array}
\right)^{-1}
\left(\begin{array}{c}
    \frac{1}{2}   \\
     0 \\
     \frac{1}{2} \\
     0 \\
     1 \\
     1
\end{array}
\right).
$$
The second example is $u_1 + u_1^{-1}$. Since we have for unitary $U_1$ in any $\ast$-algebra 
$$U_1 + U_1^{-1} =
\left(
\begin{array}{ccc}
    1 & 0 & 1 
\end{array}
\right)
\left(
\begin{array}{ccc}
    1 & -U_1 & 0 \\
    0 & 1 & 0 \\
    0 & 0 & U_2
\end{array}
\right)^{-1}
\left(
\begin{array}{c}
     0 \\
     1 \\
     1
\end{array}
\right),
$$
we have a formal self-adjoint linearization 
$$U_1+U_1^{-1}=
\left(
\begin{array}{cccccc}
 \frac{1}{2}&0&\frac{1}{2} &0 &1 &1 
\end{array}
\right)
\left(
\begin{array}{cccccc}
 0&0 &0 &1 & 0&0 \\
 0&0 &0 &-U_1^\ast &1&0 \\
 0&0 &0 &0 &0 & U_1^\ast \\
 1 & -U_1 & 0 & 0 & 0& 0 \\
 0 & 1 & 0 &0& 0& 0 \\
 0 & 0& U_1 & 0& 0& 0 
\end{array}
\right)^{-1}
\left(\begin{array}{c}
    \frac{1}{2}   \\
     0 \\
     \frac{1}{2} \\
     0 \\
     1 \\
     1
\end{array}
\right).
$$
\end{exam}

\subsection{Unbounded random variables}\label{subsec:unbounded}

In this subsection, we set $(\M,\tau)$ to be a tracial $W^{\ast}$-probability space (i.e., a von Neumann algebra $\M$ that is
endowed with a faithful normal tracial state $\tau:\M\rightarrow\C$).
The condition that $\tau$ is a trace is necessary since we are going to consider closed and densely defined operators affiliated with the von Neumann algebra $\M$.
We will call these operators unbounded operators.
In general, unbounded operators might not well-behave under either addition or composition.
However, in the case of tracial $W^{\ast}$-probability space, they form a $\ast$-algebra, denoted by $\widetilde{\M}$, which provides us a framework in which one has well-defined evaluations of rational expressions.

In a language of probability, this framework allows us to consider random variables that may not have compact support or even finite moments. 
For a normal random variable $X$ in a $W^\ast$-probability space $(\M,\tau)$, we know that $X$ has finite moments of all orders and its analytic distribution $\mu_X$ determined by the moments (i.e., the probability measure associated to $X$ by a representation theorem of Riesz) has a compact support.
For an (unbounded) operator $X$ in $\widetilde{\M}$, it may not have finite moments.
But we could still associate a probability measure to $X$ via the spectral theorem.
We refer the interested reader to \cite{MN36, Bla06} for more details on unbounded operators (which are also known as measurable operators as the noncommutative analog of measurable functions, cf. \cite{Ter81}).

Let $\Proj(\M)$ denote the set of self-adjoint projections in $\M$ and let $\widetilde{\M}_{sa}$ be the set of self-adjoint elements in $\widetilde{\M}$.
Given an element $X \in \widetilde{\M}_{\sa}$, for a Borel set $B$ on $\R$, we denote by $\1_B(X)\in\Proj(\M)$ the spectral projection of $X$ on $B$ given by the spectral theorem (see, for example, \cite{Con90}).
Then we can associate a probability measure $\mu_X$ to $X$ as follows. 

\begin{dfn}\label{def:analytic_distribution}
For $X \in \widetilde{\M}_{\sa}$, we define its \emph{analytic distribution} $\mu_X$ by
\[
\mu_X(B):=\tau(\1_B(X)),\quad \text{for all Borel sets $B\subseteq\R$}.
\]
Furthermore, we define the \emph{cumulative distribution function} of $X$ as the function $\F_{X}:\R \to [0,1]$ given by, 
\[
\F_{X}(t):=\int_{-\infty}^t 1d\mu_X(s)=\tau(\1_{(-\infty,t]}(X)).
\]
\end{dfn}

In particular, if we take $\M=L^\infty(\Omega,\mathbb{P})$ and $\tau=\mathbb{E}$ for some probability measure space $(\Omega,\mathbb{P})$, then $\widetilde{\M}$ is the $\ast$-algebra consisting of all measurable functions, i.e., classical random variables.
Moreover, the analytic distribution and cumulative distribution defined above coincide with their classical counterparts.

Recall that for a probability measure, $\mu$ on $\R$. A number $\lambda\in\R$ is called an \emph{atom} of $\mu$ if $\mu(\{\lambda\})\neq0$.
Thus for a random variable $X$ in $\widetilde{\M}_\sa$, we say that $\lambda\in\R$ is an atom for $X$ if $\lambda$ is an atom for $\mu_X$.
Moreover, we see that $X$ has an atom $\lambda\in\R$ if and only if $p_{\ker(\lambda-X)}\neq0$, where $p_{\ker(\lambda-X)}\in\Proj(\M)$ is the orthogonal projection onto the kernel of $\lambda-X$ (in the Hilbert space $L^2(\M,\tau)$).
For an atom $\lambda$ of $X$, we have
\[
\mu_X(\{\lambda\})=\tau(p_{\ker(\lambda-X)}).
\]

A closely related notion is a rank defined via the image.
That is, we define
\[
\rk(X):=\tau(p_{\overline{\im X}}),
\]
where $p_{\overline{\im X}}$ is the orthogonal projection onto the closure of the image of $X$.
The following alternative description of this rank will be needed later:
\begin{equation}\label{defrank}
    \rk(X) = \inf\{\tau(r) \mid r\in\Proj(\M),\ rX=X\} .
\end{equation}
Clearly, since $p_{\overline{\im(X)}}X=X$, we have $\inf\{\tau(r) \mid r\in\Proj(\M),rX=X\}\leq\tau(p_{\overline{\im(X)}})=\rk(X)$.
To see it is an equality, note that for any $r\in\Proj(\M)$ satisfying $rX=X$, $\im(X)\subseteq\im(r)$, which implies that $p_{\overline{\im(X)}} \leq r$.

\subsection{The quantity $\Delta$}\label{subsec:Delta}

The regularity condition that we impose in Theorem \ref{conver} on the limit of the considered random matrix model involves the quantity $\Delta$ which was introduced by Connes and Shlyakhtenko in \cite{CS05}. We briefly recall the definition. Let $(\M,\tau)$ be a tracial $W^\ast$-probability space and consider a tuple $x=(x_1,\dots,x_d)$ of (not necessarily self-adjoint) noncommutative random variables in $\M$. We denote by $\F(L^2(\M,\tau))$ the ideal of all finite rank operators on $L^2(\M,\tau)$ and by $J$ Tomita's conjugation operator, i.e., the conjugate-linear map $J: L^2(\M,\tau) \to L^2(\M,\tau)$ that extends isometrically the conjugation $x \mapsto x^\ast$ on $\M$. We then put
$$\Delta(x) := d - \dim_{\M \mathrel{\overline{\otimes}} \M^{\operatorname{op}}} \overline{\bigg\{(T_1,\dots,T_d) \in \F(L^2(\M,\tau))^d \mathrel{\bigg|} \sum^d_{j=1} [T_j, J x_j^\ast J] = 0\bigg\}}^{\operatorname{HS}},$$
where the closure is taken with respect to the Hilbert-Schmidt norm. Note that in contrast to \cite{CS05}, we do not require the set $\{x_1,\dots,x_d\}$ to be closed under the involution $\ast$; see also \cite{MSY19}. Despite this slight deviation from the setting of \cite{CS05}, the following result remains true.

\begin{thm}[Theorem 3.3 (e) in \cite{CS05}]
Let $1\leq k<d$ and suppose that the sets $\{x_1,\dots,x_k\}$ and $\{x_{k+1},\dots,x_d\}$ are freely independent, then
$$\Delta(x_1,\dots,x_d) = \Delta(x_1,\dots,x_k) + \Delta(x_{k+1},\dots,x_d).$$
\end{thm}

Further, we recall from \cite[Corollary 6.4]{MSY19} that $\Delta(u)=d$ for every $d$-tuple $u$ of freely independent Haar unitary elements in $(\M,\tau)$.

In the particular case of a $d$-tuple $x$ consisting of self-adjoint operators in $\M$, Corollary 4.6 in \cite{CS05} says that $d \geq \Delta(x) \geq \delta(x)$, where $\delta(x)$ denotes the so-called \emph{microstates free entropy dimension} which Voiculescu introduced in \cite[Definition 6.1]{Voi94}.
Now, if the $x_1,\dots,x_d$ are freely independent, then Proposition 6.4 in \cite{Voi94} tells us that
$$\delta(x) = d - \sum^d_{j=1} \sum_{t\in\R} \mu_{x_j}(\{t\})^2,$$
where $\mu_{x_j}$ is the analytic distribution of the operator $x_j$ in the sense of Definition \ref{def:analytic_distribution}. We infer that $\Delta(x_1,\dots,x_d) = d$ if $x_1,\dots,x_d$ are self-adjoint, freely independent and their individual analytic distributions $\mu_{x_1},\dots,\mu_{x_d}$ are all non-atomic. For reference, we summarize these observations by the following corollary.

\begin{cor}\label{cor:maximal_Delta}
Let $x=(x_1,\dots,x_{d_1})$ be a $d_1$-tuple of self-adjoint and freely independent elements in $(\M,\tau)$ with $\mu_{x_1},\dots,\mu_{x_{d_1}}$ being non-atomic. Further, let $u=(u_1,\dots,u_{d_2})$ be a $d_2$-tuple of freely independent Haar unitary elements in $(\M,\tau)$. Suppose that $x$ and $u$ are freely independent. Then $\Delta(x,u)=d_1+d_2$.
\end{cor}

\section{Evaluations of non-degenerate matrix-valued noncommutative rational expressions}\label{sec:evaluations}

By definition, every non-degenerate matrix-valued noncommutative rational expression has a non-empty domain when evaluations in matrices of sufficiently large size are considered. In this section, we show that much more is true. Namely, we establish that the assumptions of Theorem \ref{conver} are satisfied in very general situations.

\subsection{Evaluations in random matrices}

The following result asserts, roughly speaking, that one can almost surely evaluate every non-degenerated matrix-valued noncommutative rational expression in ``absolutely continuous'' random matrix models, provided that their size is large enough. The precise statement reads as follows.

\begin{thm}\label{thm:evaluation_rand_mat_abscont}
Let $R$ be a matrix-valued noncommutative rational expression in $d=d_1+d_2$ formal variables, which is non-degenerate. Suppose that $\mu_{d_1,d_2}^N$ is a probability measure on $M_N(\C)_\sa^{d_1} \times U_N(\C)^{d_2}$ which is absolutely continuous with respect to the product measure of the Lebesgue measure on $M_N(\C)_\sa$ and the Haar measure on $U_N(\C)$. If $(X^N,U^N)$ is a tuple of random matrices in $M_N(\C)_\sa^{d_1} \times U_N(\C)^{d_2}$ with law $\mu_{d_1,d_2}^N$, then there exists some $N_0\in\N$ such that almost surely $(X^N,U^N) \in \dom_{M_N(\C)}(R)$ for all $N \geq N_0$.
\end{thm}

\begin{rem}
	It is essential to work over the field of complex numbers for Theorem \ref{thm:evaluation_rand_mat_abscont} to be true. To see this, let us consider the noncommutative rational expression $r=(x_1 x_2 - x_2 x_1)^{-1}$. Note that there are real matrices $X_1,X_2$ at which one can evaluate $r$, but in $M_N(\R)$ with $N$ odd, there cannot exist symmetric real matrices $X_1,X_2$ at which the evaluation $r(X_1,X_2)$ would be defined, since necessarily $\det(X_1 X_2 - X_2 X_1)=0$. Indeed
	$$\det(X_1 X_2 - X_2 X_1) = \det\big({}^t(X_1 X_2 - X_2 X_1)\big) = - \det(X_1 X_2 - X_2 X_1).$$
	This observation is consistent with the proof of Proposition \ref{prop:identity_principle}, on which Theorem \ref{thm:evaluation_rand_mat_abscont} relies since we use complex analysis techniques.
	
	One can also find an algebraic construction of a symmetric matrix in the domain of a noncommutative rational expression in \cite[Remark 6.7]{xx9}.
\end{rem}

The proof of Theorem \ref{thm:evaluation_rand_mat_abscont} relies on a study of evaluations of affine linear pencils. The first step is the following proposition, which requires some notation. Consider an affine linear pencil
\begin{equation}\label{eq:linear_pencil}
Q = A_0 \otimes 1 + \sum^{d_1}_{j=1} A_j \otimes x_j + \sum^{d_2}_{j=1} B_j \otimes u_j
\end{equation}
in the variables $x=(x_1,\dots,x_{d_1})$ and $u=(u_1,\dots,u_{d_2})$, say with coefficients $A_0,A_1,\dots,A_{d_1}$ and $B_1,\dots,B_{d_2}$ taken from $M_k(\C)$. We regard $Q$ as an element in
$$M_k(\C) \otimes \C\langle x,u \rangle \cong M_k(\C\langle x,u \rangle).$$
Given an $d$-tuple $Z = (Z',Z'')$ of matrices in $M_N(\C)$, we consider the evaluation of $Q$ at $Z$ which is given by
$$Q(Z) := A_0 \otimes 1 + \sum^{d_1}_{j=1} A_j \otimes Z_j' + \sum^{d_2}_{j=1} B_j \otimes Z_j'',$$
where $Q(Z)$ lies in $M_k(\C) \otimes M_N(\C) \cong M_{kN}(\C)$.
Building on such evaluations, we associate $Q$ with functions
$$\phi^{(N)}_Q:\ M_N(\C)^d \longrightarrow \C,\quad Z \longmapsto \det(Q(Z))$$
for every $N\in\N$. Notice that $\phi^{(N)}_Q$ is a holomorphic commutative polynomial in the $d N^2$ complex matrix entries appearing in the tuple $Z$. This allows us to use the machinery of complex analysis to relate $\phi^{(N)}_Q$ with its restriction to the real space $M_N(\C)^{d_1}_\sa \times U_N(\C)^{d_2}$.

\begin{prop}\label{prop:identity_principle}
Let $Q$ be an affine linear pencil of the form \eqref{eq:linear_pencil} in $M_k(\C) \otimes \C\langle x,u \rangle$ and let $N\in\N$. If $\phi^{(N)}_Q|_{M_N(\C)^{d_1}_\sa \times U_N(\C)^{d_2}} \equiv 0$, then $\phi^{(N)}_Q \equiv 0$.
\end{prop}

\begin{proof}
Fix any $Z=(Z',Z'') \in M_N(\C)^{d_1} \times M_N(\C)^{d_2}$ and suppose that the $d_2$-tuple $Z''$ consists of invertible matrices.
We write $Z' = X + iY$ with the tuples $X=(X_1,\dots,X_{d_1}),Y=(Y_1,\dots,Y_{d_1}) \in M_N(\C)^{d_1}_\sa$ that are given by $X_j := \Re(Z_j')$ and $Y_j := \Im(Z_j')$ for $j=1,\dots,d_1$.
Further, for $j=1,\dots,d_2$, we consider the polar decomposition $Z_j'' = P_j U_j$ of $Z_j''$ with a positive definite matrix $P_j\in M_N(\C)$ and $U_j \in U_N(\C)$.
As the matrices $P_1,\dots,P_{d_2}$ are positive definite, we can define a holomorphic function $f: \C \to \C$ by
\begin{multline*}
f(z) := \phi^{(N)}_Q\big(X_1 + z Y_1, \dots, X_{d_1} + z Y_{d_1},\\ \exp(-i z \log(P_1)) U_1, \dots, \exp(-i z \log(P_{d_2})) U_{d_2}\big)
\end{multline*}
for $z\in\C$.
Due to the assumption that $\phi^{(N)}_Q|_{M_N(\C)^{d_1}_\sa \times U_N(\C)^{d_2}} \equiv 0$, we have that $f|_\R \equiv 0$. Thus, by the identity principle, it follows that $f$ vanishes identically on $\C$. In particular, $\phi^{(N)}_Q(Z) = f(i) = 0$. This shows that $\phi^{(N)}_Q$ vanishes on all $d$-tuples $Z=(Z',Z'') \in M_N(\C)^{d_1} \times M_N(\C)^{d_2}$ satisfying the condition that $Z''$ consists of invertible matrices. Since those are dense in $M_N(\C)^d$, the assertion follows.
\end{proof}

With the help of Proposition \ref{prop:identity_principle}, we see that fullness of affine linear pencils $Q$ can be detected by evaluations of $Q$ at points in $M_N(\C)^{d_1}_\sa \times U_N(\C)^{d_2}$.

\begin{prop}\label{prop:full_sa_evaluation}
Let $Q$ be an affine linear pencil of the form \eqref{eq:linear_pencil} in $M_k(\C) \otimes \C\langle x,u \rangle$. Then the following statements are equivalent.
\begin{enumerate}
    \item $Q$ is full.
\item there exists $N_0\in\N$ with the following property: for each $N\geq N_0$, we have that $\phi^{(N)}_Q|_{M_N(\C)^{d_1}_\sa \times U_N(\C)^{d_2}} \not\equiv 0$, i.e., one can find some $d$-tuple $(X^N,U^N) \in M_N(\C)_\sa^{d_1} \times U_N(\C)^{d_2}$ for which $Q(X^N,U^N)$ becomes invertible in $M_{kN}(\C)$.
\end{enumerate}
\end{prop}

\begin{proof}
If $Q$ is not full, we have a non-trivial factorization of $Q$, and its evaluation $Q(X^N,U^N)$ cannot be invertible. This implies $(ii)\implies(i)$.

First, we note that there exists some $N_0\in\N$ such that $\phi^{(N)}_Q \not\equiv 0$ for all $N\geq N_0$. This fact is well-known (see Proposition 2.4 in \cite{Vol2021}, for instance), but we include the argument for the sake of completeness. Since $Q$ is full, $Q(x)$ is invertible as a matrix over the free skew field $\C\lff x,u \rff$; see Remark \ref{rem:linear-full} \ref{it:invertible_full}. Its inverse $Q(x)^{-1} \in M_k(\C\lff x,u\rff)$ is represented by some non-degenerate $k \times k$ matrix-valued noncommutative rational expression $R$, i.e., we have $Q^{-1}(x) = R(x)$; this follows by applying Remark \ref{rem:expressions_vs_functions} entrywise. From Theorem \ref{thm:matrix_domain_large_dimension}, we know that there exists some $N_0\in\N$ such that $\dom_{M_N(\C)}(R) \neq \emptyset$ for all $N\geq N_0$.
Thanks to Remark \ref{rem:representing_rational_expression}, the identity $Q(x,u) R(x,u) = I_k$ over $\C\lff x,u\rff$ continues to hold on $\dom_{M(\C)}(R)$, and by applying determinants, we infer that $\phi^{(N)}_Q \not\equiv 0$ for all $N\geq N_0$, as desired.

Having this, Proposition \ref{prop:identity_principle} guarantees that $\phi^{(N)}_Q$ does not vanish identically on all of $M_N(\C)_\sa^{d_1} \times U_N(\C)^{d_2}$, as we wished to show.
\end{proof}

In the next step, we involve the concrete random matrix model we want to consider.

\begin{prop}\label{prop:invertibility_rand_mat_abscont}
Let $Q$ be an affine linear pencil of the form \eqref{eq:linear_pencil} in $M_k(\C) \otimes \C\langle x,u \rangle$  which is full. For $N\in\N$, let $(X^N,U^N)$ be a random matrix in $M_N(\C)_\sa^{d_1} \times U_N(\C)^{d_2}$ with an absolutely continuous law $\mu^N_{d_1,d_2}$ like in Theorem \ref{thm:evaluation_rand_mat_abscont}. Then there exists $N_0\in\N$ such that almost surely $Q(X^N,U^N)$ is invertible in $M_k(\C) \otimes M_N(\C) \cong M_{kN}(\C)$ for all $N \geq N_0$.
\end{prop}

\begin{proof}
Thanks to Proposition \ref{prop:full_sa_evaluation}, since $Q$ is assumed to be full, there is an $N_0\in\N$ such that none of the functions $\phi^{(N)}_A|_{M_N(\C)^{d_1}_\sa \times U_N(\C)^{d_2}}$ for $N\geq N_0$ can vanish identically.
Notice that $M_N(\C)_\sa^{d_1} \times U_N(\C)^{d_2}$ is a real manifold of dimension $d N^2$. In suitable local charts, we see that $\phi^{(N)}_Q|_{M_N(\C)^{d_1}_\sa \times U_N(\C)^{d_2}}$ induces a real analytic function on an open subset of $\R^{d N^2}$ and can therefore vanish only on a set of Lebesgue measure $0$. Due to the choice of $\mu_{d_1,d_2}^N$, we conclude that, for each $N\geq N_0$, the random matrix $Q(X^N,U^N)$ is almost surely invertible in $M_{kN}(\C)$.
\end{proof}

\begin{proof}[Proof of Theorem \ref{thm:evaluation_rand_mat_abscont}]
We define the set $\mathfrak{R}_0$ of all non-degenerate matrix-valued noncommutative rational expressions $R$ for which the conclusion of Theorem \ref{thm:evaluation_rand_mat_abscont} is true, i.e., there exists $N_0\in\N$ such that almost surely $(X^N,U^N) \in \dom_{M_N(\C)}(R)$ for all $N \geq N_0$. We have to prove that $\mathfrak{R}_0$ consists, in fact, of all non-degenerate matrix-valued noncommutative rational expressions.

Notice that all matrix-valued noncommutative polynomial expressions obviously belong to $\mathfrak{R}_0$. Further, it is easily seen that both $R_1 + R_2$ and $R_1 \cdot R_2$ are in $\mathfrak{R}_0$ whenever we take $R_1,R_2 \in \mathfrak{R}_0$ for which the respective arithmetic operation makes sense. Therefore, it only remains to prove that if $R\in\mathfrak{R}_0$ is square and enjoys the property that $R^{-1}$ is non-degenerate, then necessarily $R^{-1}\in \mathfrak{R}_0$. To verify this, we take any square matrix-valued noncommutative rational expression $R$ belonging to $\mathfrak{R}_0$ for which $R^{-1}$ is non-degenerate. Further, let $\rho=(v,Q,w)$ be a formal linear representation of $R$ in the sense of Definition \ref{def:formal_linear_representation}, say of dimension $k$; see Theorem \ref{linear}.

By assumption, we have that $R^{-1}$ is a non-degenerate matrix-valued noncommutative rational expression. Thus, Remark \ref{rem:linear-full} \ref{it:display_full} gives us that the affine linear pencil in $d$ variables with coefficients from $M_{k+p}(\C)$ which is given by
$$\tilde{Q} := \begin{pmatrix} 0_{p \times p} & v\\ w & Q\end{pmatrix}$$
is full. Therefore, Proposition \ref{prop:invertibility_rand_mat_abscont} tells us that an $N_0\in\N$ exists such that almost surely $\tilde{Q}(X^N,U^N)$ is invertible in $M_{(k+p)N}(\C)$ for all $N \geq N_0$.
Since $R\in\mathfrak{R}_0$, we may suppose that (after enlarging $N_0$ if necessary) that at the same time almost surely $(X^N,U^N) \in \dom_{M_N(\C)}(R)$ for all $N\geq N_0$. Because $\rho$ is a formal linear representation, the latter implies that almost surely $Q(X^N,U^N)$ is invertible and $R(X^N,U^N) = v Q(X^N,U^N)^{-1} w$ for all $N\geq N_0$.
Putting these observations together, again with the help of the Schur complement formula, we see that almost surely $R(X^N,U^N)$ is invertible for all $N\geq N_0$. In other words, we have almost surely that $(X^N,U^N) \in \dom_{M_N(\C)}(R^{-1})$ for all $N\geq N_0$. The latter means that $R^{-1}\in\mathfrak{R}_0$, as desired.
\end{proof}

\subsection{Evaluation in operators with maximal $\Delta$}

It follows from \cite[Theorem 1.1]{MSY19} that for any $d$-tuple $X=(X_1,\dots,X_d)$ of (not necessarily self-adjoint) operators in some $W^\ast$-probability space $(\M,\tau)$ which satisfy the ``regularity condition'' $\Delta(X)=d$, where $\Delta$ stands for a quantity that was introduced in \cite{CS05} and which we discussed in Section \ref{subsec:Delta}, then the canonical evaluation homomorphism
$$\mathrm{ev}_X:\ \C\langle x_1,\dots,x_d\rangle \to \M$$
which is determined by $1\mapsto 1$ and $x_j \mapsto X_j$ for $j=1,\dots,d$ extends to an injective homomorphism
$$\mathrm{Ev}_X:\ \FF \to \widetilde{\M}$$
into the $\ast$-algebra $\widetilde{\M}$ of all closed and densely defined operators affiliated with $\M$; see Section \ref{subsec:unbounded}.

While the result of \cite{MSY19} addresses evaluations of noncommutative rational functions, it leaves open the question of whether also all non-degenerate rational expressions can be evaluated; indeed, this is not immediate as the domain of a rational function is larger than the domain of any of its representing non-degenerate noncommutative rational expressions. This question is answered in the affirmative by the next theorem, which gives the conclusion even in the matrix-valued case. For that purpose, we will consider the canonical amplifications
$$\mathrm{Ev}_X^{\bullet}:\ M_{\bullet}(\FF) \to M_{\bullet}(\widetilde{\M}).$$

\begin{thm}\label{thm:evaluation_maximalDelta}
Let $X=(X_1,\dots,X_d)$ be a $d$-tuple of (not necessarily self-adjoint) operators in some tracial $W^\ast$-probability space $(\M,\tau)$ satisfying $\Delta(X)=d$. Then, for every non-degenerate matrix-valued noncommutative rational expression $R$, we have that $X \in \dom_{\widetilde{\M}}(R)$ and $R(X) = \mathrm{Ev}_X^\bullet(R(x))$, where $R(x)$ is the matrix over $\FF$ associated to $R$ via Lemma \ref{lem:representing_rational_expression}.
\end{thm}

\begin{proof}
The proof is similar to the proof of Theorem \ref{thm:evaluation_rand_mat_abscont}. Here, we consider the set $\mathfrak{R}_0$ of all non-degenerate matrix-valued noncommutative rational expressions $r$ for which the conclusion of Theorem \ref{thm:evaluation_maximalDelta} is true, i.e., we have $X \in \dom_{\widetilde{\M}}(R)$ and $R(X) = \mathrm{Ev}_X^\bullet(R(x))$. We want to show that $\mathfrak{R}_0$ consists of all non-degenerate matrix-valued noncommutative rational expressions. This can be done in almost the same way as in Theorem \ref{thm:evaluation_rand_mat_abscont}, except for some slight modifications in the last step. Suppose that $R\in\mathfrak{R}_0$ is of size $p\times p$ and has the property that $R^{-1}$ is non-degenerate. Consider a formal linear representation $\rho=(u,A,v)$ of $r$, say of dimension $k$. Like in the proof of Theorem \ref{thm:evaluation_rand_mat_abscont}, we deduce from Remark \ref{rem:linear-full} \ref{it:display_full} that the associated affine linear pencil
$$\tilde{A} := \begin{pmatrix} 0 & u\\ v & A\end{pmatrix}$$
is full. Now, by applying \cite[Theorem 5.6]{MSY19} instead of Proposition \ref{prop:invertibility_rand_mat_abscont}, we get that $\tilde{A}(X)$ is invertible. Having this, we can proceed again like in the proof of Theorem \ref{thm:evaluation_rand_mat_abscont}, and we arrive at $X \in \dom_{\widetilde{\M}}(R^{-1})$. Moreover, since $R(X) = \mathrm{Ev}_X^\bullet(R(x))$ by the assumption $R\in\mathfrak{R}_0$, we further get that $R^{-1}(X) = R(X)^{-1} = \mathrm{Ev}_X^\bullet(R(x))^{-1} = \mathrm{Ev}_X^\bullet(R(x)^{-1}) = \mathrm{Ev}_X^\bullet(R^{-1}(x))$; notice that $R(x)$ is invertible because Lemma \ref{lem:representing_rational_expression} guarantees that $x \in \dom_{\C\lff x\rff}(R^{-1})$ as $R^{-1}$ was assumed to be non-degenerate. In summary, we see that $R^{-1} \in \mathfrak{R}_0$.
\end{proof}

\section{Convergence in law of the spectral measure}\label{sec:main}

\subsection{Estimate on the cumulative distribution function of the spectral measure of self-adjoint operators} 

In this subsection, we list and prove a few properties that we need in the following subsection to prove Theorem \ref{conver} about the convergence of the empirical measure of a self-adjoint non-degenerate matrix-valued noncommutative rational expression evaluated in matrices towards the analytic distribution of the limiting operator. 

\begin{lm}[Lemma 3.2 in \cite{BV93}]\label{egal}
For $X \in \widetilde{\M}_{\mathrm{sa}}$ and $t \in \R$ we have
$$\F_X(t) = \mathrm{max}\{\tau(p) \mid  p \in \Proj(\M), \ p (t-X) p \geq 0 \}.$$
\end{lm}

The crux of the proof of Theorem \ref{conver} lies in the following two lemmas.

\begin{lm}
\label{rang}
	Let $X,Y\in  \widetilde{\M}_{\mathrm{sa}}$, then
	$$ \sup_{t\in\R} |\F_X(t) - \F_{X+Y}(t)| \leq \rk(Y).$$
\end{lm}

\begin{proof}
	We fix $t\in\R$. Let $r\in\Proj(\M)$ be such that $rY=Y$ and $q\in\Proj(\M)$ such that $q(t-X)q \geq 0$. Then if we set $p = q\wedge(1-r)$, we have $pY = 0$ and $pq =p$, thus
	$$ p(t-X-Y)p = p(t-X)p = pq(t-X)qp \geq 0 .$$
		
	\noindent Consequently,
	$$ \F_{X+Y}(t)\geq \tau(p) \geq \tau(q) - \tau(r) .$$
	By taking the supremum over $q$ and the infimum over $r$, we get that
	$$ \F_{X+Y}(t)\geq \F_{X}(t) - \rk(Y) .$$
	
	Now let's assume that $q$ is such that $q(t-X-Y)q \geq 0$, then similarly with $p = q\wedge(1-r)$,
	$$ p(t-X)p = p(t-X-Y)p = p q (t-X-Y)p q \geq 0 .$$
	Hence 
	$$ \F_X(t) \geq \tau(p) \geq \tau(q) - \tau(r) .$$
	And once again, by taking the supremum over $q$ and the infimum over $r$, we get that
	$$ \F_X(t) \geq \F_{X+Y}(t) - \rk(Y) .$$
	Hence the conclusion.
\end{proof}

The authors are indebted to Mikael de la Salle for indicating to them the following lemma.

\begin{lm}
\label{decr}
	Let $p\in \Proj(\M)$, $X\in\widetilde{\M}_{sa}$, then $ \rk(pXp) \leq \rk(X) $.
\end{lm}

\begin{proof}
	Let $q\in\Proj(\M)$ be such that $qX=X$, $r = p\wedge (1-q)$, then $r + 1-p$ is such that 
	$$ (r + 1-p) pXp  = rpXp = rXp = rqXp = 0 .$$
	Consequently, $ (p-r) pXp = pXp $. And since $p\geq r$, $p-r$ is a self-adjoint projection, hence
	$$ \rk(pXp) \leq \tau(p-r) \leq \tau(q) .$$
	Hence the conclusion is obtained by taking the infimum over $q$.
\end{proof}

\subsection{Main result}

This subsection focuses on proving the convergence in law of the empirical measure of matrix-valued noncommutative rational expressions evaluated in matrices under some assumptions. Theorem \ref{conver} is for deterministic matrices, but it can easily be extended to random matrices by applying this result almost surely. 

\begin{thm}\label{conver}
Let $X^N=(X_1^N,\dots,X_{d_1}^N)$ be a $d_1$-tuple of deterministic self-adjoint matrices and let $U^N=(U_1^N,\dots,U_{d_2}^N)$ be a $d_2$-tuple of deterministic unitary matrices. Further, let $R$ be a non-degenerate square matrix-valued noncommutative rational expression in $d=d_1+d_2$ variables which is self-adjoint of type $(d_1,d_2)$ in the sense of Definition \ref{def:rational_expression_selfadjoint}. Suppose that the following conditions are satisfied:
	\begin{enumerate}
	    \item $(X^N,U^N)$ converges in $\ast$-distribution towards a $d$-tuple of noncommutative random variables $(x,u)$ in some tracial $W^\ast$-probability space $(\M,\tau)$ satisfying $\Delta(x,u)=d$.
	    \item\label{it:conver_evaluation} For $N$ large enough $R(X^N,U^N)$ is well-defined, i.e., there exists $N_0\in\N$ such that $(X^N,U^N) \in \dom_{M_N(\C)}(R)$ for all $N\geq N_0$.
	\end{enumerate}
	Then $(x,u) \in \dom_{\widetilde{\M}}(R)$, so that $R(x,u)$ is well-defined, and the empirical measure of $R(X^N,U^N)$ converges in law towards the analytic distribution of $R(x,u)$.
\end{thm}

The fact that $(x,u) \in \dom_{\widetilde{\M}}(R)$ holds was established already in Theorem \ref{thm:evaluation_maximalDelta}. Accordingly, the main statement of Theorem \ref{conver} is the convergence of the empirical measure of $R(X^N,U^N)$ towards the spectral measure of $R(x,u)$.
This convergence result actually holds in a more general setting than the above theorem.
We summarize it as the following proposition.

\begin{prop}\label{prop:convergence}
For each $N\in\N$, let $X^N=(X_1^N,\dots,X_{d}^N)$ be a $d$-tuple of (possibly unbounded) operators affiliated with some tracial $W^\ast$-probability space $(\M^{(N)},\tau^{(N)})$.
Furthermore, for any $\ast$-polynomial $P$ we assume $P(X^N,X^{N\ast})\in L^1(\M^{(N)},\tau^{(N)})$ and $\tau^{(N)}(P(X^N,X^{N\ast}))$ converges towards $\tau(P(X,X^*))$ where $X=(X_1, \dots,$ $X_d)$ is a $d$-tuple of noncommutative (bounded) random variables in some tracial $W^\ast$-probability space $(\M,\tau)$.
Let $R$ be a square matrix-valued noncommutative rational expression in $d$ variables such that, for all $N\in\N$ which are sufficiently large,
\begin{enumerate}
\item\label{it:convergence_eval} $X^N \in\dom_{\widetilde{\M}^{(N)}}(R)$ and $X \in\dom_{\widetilde{\M}}(R)$,
\item\label{it:convergence_sa} $R(X^N)$ and $R(X)$ are self-adjoint.
\end{enumerate}
Then the analytic distribution of $R(X^N)$ converges in law towards the analytic distribution of $R(X)$.
\end{prop}

Once Proposition \ref{prop:convergence} is shown, the statement about the convergence in Theorem \ref{conver} follows immediately. Indeed, the condition formulated in Item \ref{it:convergence_eval} of Proposition \ref{prop:convergence} is satisfied as we have $(X^N,U^N) \in \dom_{M_N(\C)}(R)$ for all $N\geq N_0$ by Item \ref{it:conver_evaluation} of Theorem \ref{conver} and $(x,u) \in \dom_{\widetilde{\M}}(R)$ by Theorem \ref{thm:evaluation_maximalDelta}; further, we have that $R(X^N,U^N)$ for all $N\geq N_0$ and $R(x,u)$ are self-adjoint thanks to Definition \ref{def:rational_expression_selfadjoint} as $R$ is supposed to be self-adjoint of type $(d_1,d_2)$, so that the condition in Item \ref{it:convergence_sa} of Proposition \ref{prop:convergence} is fulfilled as well.

Let us outline the proof of Proposition \ref{prop:convergence}.
Let $\rho=(Q,w)$ be a self-adjoint formal linear representation of $R$ in the sense of Definition \ref{def:formal_linear_representation_selfadjoint} which is moreover proper as given by Theorem \ref{sa-lin}.
Thanks to Lemma \ref{rang}, we can ignore the singularity in $0$ of $Q(X,X^\ast)^{-1}$.
More precisely, as long as the spectral measure of $Q(X,X^\ast)$ has no atom at $0$, we can use Lemma \ref{rang} to prove that the cumulative distribution function of $w^* Q(X,X^\ast)^{-1} w$ is close to the one of $w^* f_{\varepsilon}(Q(X,X^{\ast})) w$ where $f_{\varepsilon}$ is a continuous function which is equal to $t\mapsto t^{-1}$ outside of a neighborhood of $0$ of size $\varepsilon$.
Then we can use the convergence in $\ast$-distribution of $X^N$ to show that the cumulative distribution function of $w^* f_{\varepsilon}(Q(X^N,X^{N\ast})) w$ converges towards the correct limit when we let $N$ go to infinity and $\varepsilon$ go to $0$.

It is important to note that in this subsection, by convergence in $\ast$-distribution of $X^N$ of noncommutative random variables $X^N=(X_1^N,\dots,X_{d}^N)$, we mean that the trace of any noncommutative $\ast$-polynomial $P$ evaluated in $X^N$ converges towards the trace of $P(X,X^\ast)$ where $X$ is a $d$-tuple of noncommutative random variables in some tracial $W^\ast$-probability space.
In particular, this does \emph{not} exclude the case where the operator norm of $X_i^N$ is not bounded over $N$.
This forces us to do a few more computations since the convergence in law of the analytic measure of $P(X^N,X^{N\ast})$ towards the analytic measure of the limiting operator, while still valid, is not immediate anymore.

\begin{proof}[Proof of Proposition \ref{prop:convergence}]
    Let $\rho=(Q,w)$ be a proper self-adjoint formal linear representation (of dimension $k$) of $R$. If $p \in \N$ is the size of $R$, then since $k\geq p$ and $w$ have full rank, there exists a matrix $T \in \mathrm{GL}_k(\C)$ such that $w = T w_0$ where $w_0 \in M_{k\times p}(\C)$ is the rectangular matrix whose diagonal coefficients are all $1$, and non-diagonal coefficients are all $0$. By replacing $Q$ by $T^{\ast}Q T$, one can assume without loss of generality that 
    $w = w_0$.
    
    Notice that by assumption $Q_N := Q(X^N,X^{N\ast})$ is invertible in $\widetilde{\M}^{(N)}$ and $R(X^N) = w^* Q_N^{-1} w$.
    Further, we have also that $Q_\infty := Q(X,X^\ast)$ is invertible in $\widetilde{\M}$ and $R(X) = w^* Q_\infty^{-1} w$.
    Thanks to L\'evy's continuity theorem, to prove the convergence in law of the empirical measure of $R(X^N)$ towards the analytical distribution of $R(X)$, we only need to prove that the cumulative distribution function of the empirical measure of $R(X^N)$ converges towards the one of the analytical distribution of $R(X)$ at every point of continuity. I.e., we need to show that $\F_{w^\ast Q_N w}(t)$ converges towards $\F_{w^\ast Q_\infty w}(t)$ for $t\in\R$ such that the function $s\mapsto\F_{w^\ast Q_\infty w}(s)$ is continuous in $t$.
    To do so, let $g:t\mapsto t^{-1}$ and $f_{\varepsilon}:\R\to\R$ be a continuous function such that on the complementary set of $[-\varepsilon,\varepsilon]$, $f_{\varepsilon}=g$. We have for any $t\in\R$,
	\begin{align*}
		\left| \F_{w^* Q_N^{-1} w}(t) - \F_{w^* Q_\infty^{-1} w}(t) \right|
   \leq & \left| \F_{w^* f_{\varepsilon}(Q_N) w}(t) - \F_{w^* f_{\varepsilon}(Q_\infty) w}(t) \right|  \\
		&+ \left| \F_{w^* Q_N^{-1} w}(t) - \F_{w^* f_{\varepsilon}(Q_N) w}(t) \right| \\
		&+ \left| \F_{w^* Q_\infty^{-1} w}(t) - \F_{w^* f_{\varepsilon}(Q_\infty) w}(t) \right| .
	\end{align*}
	
	\noindent Thus thanks to Lemma \ref{rang},
	\begin{align*}
	\left| \F_{w^* Q^{-1}_N w}(t) - \F_{w^* Q^{-1}_\infty w}(t) \right|
    \leq & \left| \F_{w^* f_{\varepsilon}(Q_N) w}(t) - \F_{w^* f_{\varepsilon}(Q_\infty) w}(t) \right| \\
	&+ \rk(w^* (f_{\varepsilon}-g)(Q_N) w) \\
	&+ \rk(w^* (f_{\varepsilon}-g)(Q_\infty) w) .
	\end{align*}
	
	\noindent Since $w = w_0$, we have that for any $X \in M_p(\widetilde{\M})$,
	\begin{eqnarray*}
    \rk(w X w^*) &=& \rk\begin{pmatrix} X  & 0_{p \times (k-p)} \\ 0_{(k-p) \times p} & 0_{k-p}\end{pmatrix}\\
	&=& \frac{p}{k} \rk(X).
	\end{eqnarray*}
	
	\noindent This implies that
 	$$	\rk(w^* (f_{\varepsilon}-g)(Q_\infty)w) = \frac{k}{p}\times \rk(ww^* (f_{\varepsilon}-g)(Q_\infty) ww^*) \leq \frac{k}{p} \times  \rk((f_{\varepsilon}-g)(Q_\infty)) ,$$
	where in the last inequality, we used Lemma \ref{decr}. Besides $\mathbf 1_{[-\varepsilon,\varepsilon]}(Q_\infty)$ is a self-adjoint projection such that $\mathbf 1_{[-\varepsilon,\varepsilon]}(Q_\infty) (f_{\varepsilon}-g)(Q_\infty) = (f_{\varepsilon}-g)(Q_\infty)$. Consequently, with $\tr_k$ the non-renormalized trace on $M_k(\C)$ and $\tau$ the trace on $\M$,
	$$ \rk(w^* (f_{\varepsilon}-g)(Q_\infty) w) \leq \frac{1}{p}\tr_k\otimes\tau(\ \mathbf 1_{[-\varepsilon,\varepsilon]}(Q_\infty)\ ) .$$
	
	\noindent Let $h_{\varepsilon}$ be a continuous function which takes value $1$ on $[-\varepsilon,\varepsilon]$, $0$ outside of $[-2\varepsilon,2\varepsilon]$ and in $[0,1]$ elsewhere, then 
	\begin{equation}
	\label{ranksetimate}
		\rk(w^* (f_{\varepsilon}-g)(Q_\infty) w) \leq \frac{1}{p}\tr_k\otimes\tau(\ h_{\varepsilon}(Q_\infty)\ ) .
	\end{equation}
	
	\noindent Hence by applying the same reasoning to $Q_N$, we get after combining those estimates that
	\begin{align*}
	\left| \F_{w^* Q^{-1}_N w}(t) - \F_{w^* Q^{-1}_\infty w}(t) \right|
	 \leq & \left| \F_{w^* f_{\varepsilon}(Q_N) w}(t) - \F_{w^* f_{\varepsilon}(Q_\infty) w}(t) \right| \\
	& + \frac{1}{p}\tr_k\otimes \tau(\ h_{\varepsilon}(Q_\infty)\ ) \\
	& + \frac{1}{p} \tr_k\otimes\tau^{(N)}(\ h_{\varepsilon}(Q_N)\ ).
	\end{align*}
	
    To use the Portmanteau theorem, we want to prove that the analytic distribution of $w^* f_{\varepsilon}(Q_N) w$ converges towards the analytic distribution of $w^* f_{\varepsilon}(Q_\infty) w$.
    However, since this self-adjoint operator is uniformly bounded over $N$, we need to prove the convergence of the moments. That is, that
	$$ \lim_{N\to\infty} \frac{1}{p}\tr_p \otimes \tau^{(N)}\left( (w^* f_{\varepsilon}(Q_N) w)^l \right) = \frac{1}{p}\tr_p \otimes \tau\left( (w^* f_{\varepsilon}(Q_\infty) w)^l \right)$$
	for any $l$.
	The strategy consists in approximating $f_{\varepsilon}$ by a polynomial and then using the convergence in $\ast$-distribution of $X^N$.
	However, the fact that we did not assume the operator norm of $X_i^N$ to be bounded over $N$ forces us to make additional estimates.
	
	Let $C = \norm{Q_\infty}+1$, and $h$ be a non-negative continuous function which takes value $0$ on $[-C,C]$, $1$ outside of $[-C-1,C+1]$ and in $[0,1]$ elsewhere. Let $P_m$ be a polynomial such that $\norm{f_{\varepsilon}-P_m}_{\mathcal{C}^0([-C-1,C+1])} \leq 1/m$. We set
	$$
	 B^N := \big( f_{\varepsilon}(Q_N)-P_m(Q_N) \big) ( 1 - h(Q_N) ) \quad\text{and}\quad C^N := \big( f_{\varepsilon}(Q_N)-P_m(Q_N) \big) h(Q_N),
	$$
	then
	\begin{align*}
	&\frac{1}{p}\tr_p \otimes \tau^{(N)}\left( (w^* f_{\varepsilon}(Q_N) w)^l \right) - \frac{1}{p}\tr_p \otimes \tau^{(N)}\left( (w^* P_m(Q_N) w)^l \right) \\
	&= \sum_{i=1}^l \frac{1}{p}\tr_p \otimes \tau^{(N)}\left( (w^* f_{\varepsilon}(Q_N) w)^{i-1} w^*(B^N+C^N)w (w^* P_m(Q_N) w)^{l-i} \right).
	\end{align*}
	
	\noindent Thanks to the Cauchy-Schwarz inequality, we have for any $i$,
	\begin{align*}
	    &\left| \frac{1}{p}\tr_p \otimes \tau^{(N)}\left( (w^* f_{\varepsilon}(Q_N) w)^{i-1} w^*(B^N+C^N)w (w^* P_m(Q_N) w)^{l-i} \right) \right| \\
	    &\leq \left(\sqrt{\frac{1}{p}\tr_p\otimes\tau^{(N)}\left(w^*B^Nw w^* B^N w\right)} + \sqrt{\frac{1}{p}\tr_p\otimes\tau^{(N)}\left(w^*C^Nw w^* C^N w\right)}\right) \\
	    &\quad \times \sqrt{\frac{1}{p}\tr_p \otimes \tau^{(N)}\left( (w^* P_m(Q_N) w)^{2(l-i)} (w^* f_{\varepsilon}(Q_N) w)^{2(i-1)} \right)} \\
	    &\leq \left(\sqrt{\tr_k\otimes\tau^{(N)}\left((B^N)^2\right)} + \sqrt{\tr_k\otimes\tau^{(N)}\left((C^N)^2\right)}\right) \\
	    &\quad \times \left(\frac{1}{p}\tr_p \otimes \tau^{(N)}\left( (w^* P_m(Q_N) w)^{4(l-i)}\right)\right)^{1/4}\\
	    &\quad \times \left(\frac{1}{p}\tr_p \otimes \tau^{(N)}\left( (w^* f_{\varepsilon}(Q_N) w)^{4(i-1)} \right)\right)^{1/4}.
	\end{align*}
	
	\noindent Since $f_{\varepsilon}$ is bounded by a constant $K$, we have that 
	$$ \frac{1}{p}\tr_p \otimes \tau^{(N)}\left( (w^* f_{\varepsilon}(Q_N) w)^{4(i-1)} \right) \leq K^{4(i-1)} .$$
	Thanks to the convergence in $\ast$-distribution of $X^N$, and since the expression $w^* P_m(Q_N) w$ is a matrix of polynomials in $X^N$, we have
	$$ \lim_{N\to\infty} \frac{1}{p}\tr_p \otimes \tau^{(N)}\left( (w^* P_m(Q_N) w)^{4(l-i)}\right) = \frac{1}{p}\tr_p \otimes \tau\left( (w^* P_m(Q_\infty) w)^{4(l-i)}\right) ,$$
	which means that
	$$ \lim_{N\to\infty} \frac{1}{p}\tr_p \otimes \tau^{(N)}\left( (w^* P_m(Q_N) w)^{4(l-i)}\right) \leq (K+1/m)^{4(l-i)} .$$
	We also have
	$$ \tr_k\otimes\tau^{(N)}\left((B^N)^2\right) \leq \frac{k}{m}. $$
	Finally since $f_{\varepsilon}$ is bounded, there exists an integer $g$ such that for any $t\in\R$, $|f_{\varepsilon}(t)-P_m(t)| \leq (1+t^2)^g$, thus for any $r\geq 0$,
	\begin{align*}
	    \tr_k\otimes\tau^{(N)}\left((C^N)^2\right) &\leq \frac{\tr_k\otimes\tau^{(N)}\left((C^N)^2 Q_N^{2r}\right)}{N C^{2r}} \\
	    &\leq \frac{\tr_k\otimes\tau^{(N)}\left((1+Q_N^2)^{2g} Q_N^{2r}\right)}{N C^{2r}} .
	\end{align*}
	And so for any $r\geq 0$,
	\begin{align*}
	    \lim_{N\to\infty} \tr_k\otimes\tau^{(N)}\left((C^N)^2\right) &\leq \frac{\tr_k\otimes\tau\left((1+Q_\infty^2)^{2g} Q_\infty^{2r}\right)}{C^{2r}} \\
	    &\leq k \norm{(1+Q_\infty^2)^{2g}} \frac{(C-1)^{2r}}{C^{2r}}.
	\end{align*}
	\noindent So by letting $r$ go to infinity, we get 
	$$ \lim_{N\to\infty} \tr_k\otimes\tau^{(N)}\left((C^N)^2\right) = 0 .$$
	
	\noindent By combining those results, we obtain
	\begin{align*}
	&\limsup_{N\to\infty} \left| \frac{1}{p}\tr_p \otimes \tau^{(N)}\left( (w^* f_{\varepsilon}(Q_N) w)^l \right) - \frac{1}{p}\tr_p \otimes \tau\left( (w^* f_{\varepsilon}(Q_\infty) w)^l \right) \right| = \mathcal{O}(1/m).
	\end{align*}
	
	Thus, by letting $m$ go to infinity, we get the convergence of the moments. This implies that the analytic distribution of $w^* f_{\varepsilon}(Q_N) w$ converges towards the analytic distribution of $w^* f_{\varepsilon}(Q_\infty) w$.
	Thanks to Portmanteau's theorem and Lemma \ref{egal}, we have
	\begin{align*}
	    \F_{w^* f_{\varepsilon}(Q_\infty) w}(t) &\geq \limsup_{N\to\infty} \F_{w^* f_{\varepsilon}(Q_N) w}(t) \\
	    &\geq \liminf_{N\to\infty} \F_{w^* f_{\varepsilon}(Q_N) w}(t) \geq \lim_{s\to t, s< t} \F_{w^* f_{\varepsilon}(Q_\infty) w}(s) .
	\end{align*}
	
	\noindent Consequently,
	\begin{align*}
	&\limsup_{N\to\infty} \left| \F_{w^* Q_N^{-1} w}(t) - \F_{w^* Q_\infty^{-1} w}(t) \right| \\
	&\leq \lim_{s\to t, s< t} \left| \F_{w^* f_{\varepsilon}(Q_\infty) w}(t) - \F_{w^* f_{\varepsilon}(Q_\infty) w}(s) \right|
	+ \frac{2}{p} \tr_k\otimes\tau(\ h_{\varepsilon}(Q_\infty)\ ),
	\end{align*}
	where we used the convergence in $\ast$-distribution of $X^N$ once again in the last line, coupled with an argument similar to the one which let us prove the convergence of the moments of $w^* f_{\varepsilon}(Q_N) w$. But by using Lemma \ref{rang} one more time, we have
	\begin{align*}
	    &\left| \F_{w^* f_{\varepsilon}(Q_\infty) w}(t) - \F_{w^* f_{\varepsilon}(Q_\infty) w}(s) \right| \\
	    &\leq \left| \F_{w^* Q_\infty^{-1} w}(t) - \F_{w^* Q_\infty^{-1} w}(s) \right| + 2 \rk(w^* (f_{\varepsilon}-g)(Q_\infty) w) .
	\end{align*}
	
	\noindent Hence by using equation \eqref{ranksetimate}, we have that
	\begin{align*}
	    &\limsup_{N\to\infty} \left| \F_{w^* Q_N^{-1} w}(t) - \F_{w^* Q_\infty^{-1} w}(t) \right| \\
	    &\leq \lim_{s\to t, s< t} \left| \F_{w^\ast Q_\infty^{-1} w}(t) - \F_{w^\ast Q_\infty^{-1} w}(s) \right| + \frac{4}{p}\tr_k\otimes\tau(\ h_{\varepsilon}(Q_\infty)\ ).
	\end{align*}
	
	Since we made the assumption that $t$ was a point of continuity of the function $s\mapsto \F_{w^\ast Q_\infty w}(s)$, we have that $\lim_{s\to t, s\leq t} \left| \F_{w^\ast Q_\infty w}(t) - \F_{w^\ast Q_\infty w}(s) \right| = 0$.
	Besides, by the dominated convergence theorem, $\lim_{\varepsilon\to 0} \tr_k\otimes\tau( h_{\varepsilon}(Q_\infty)) = \tr_k\otimes\tau(\1_{\{0\}}(Q_\infty))$, which is equal to $0$ since otherwise the distribution of $Q_\infty$ would have an atom in $0$, in contradiction to the invertibility of $Q_\infty$; indeed, analogous to the proof of \cite[Corollary 5.13]{MSY19}, we notice that $Q_\infty \1_{\{0\}}(Q_\infty) = 0$ and conclude from the latter that since $Q_\infty$ is invertible over $\widetilde{\M}$ we necessarily have that $\1_{\{0\}}(Q_\infty) = 0$ and hence $\mu_{Q_\infty}(\{0\}) = \frac{1}{k} \tr_k\otimes\tau(\1_{\{0\}}(Q_\infty)) = 0$.
\end{proof}

\bibliographystyle{amsalpha}

\begin{thebibliography}{10}

\bibitem{Ami66}
S. A.~{Amitsur}. 
 \newblock Rational identities and applications to algebra and geometry.
 \newblock{\em Journal of Algebra},
 3(3):304--359, 1966.
 
\bibitem{AND}
G. W. Anderson.
\newblock Convergence of the largest singular value of a polynomial in independent Wigner matrices.
{\em Ann. Probab.} 41, no. 3B, 2103--2181, 2013.

\bibitem{Avitzour82}
D.~Avitzour. 
\newblock Free products of $C^*$-algebras.
\newblock {\em Trans. Amer. Math. Soc.} 271 (1982), 423-435.

\bibitem{belin-capi}
S.~Belinschi and M.~Capitaine,.
\newblock Spectral properties of polynomials in independent Wigner and deterministic matrices.
\newblock {\em Journal of Functional Analysis}, 273:3901--3973, 2016.

\bibitem{BV93}
H.~{Bercovici} and D.~{Voiculescu}.
\newblock Free Convolution of Measures with Unbounded Support.
\newblock {\em Indiana University Mathematics Journal}, 42(3):733--773
, 1993.

\bibitem{Bla06}
B. {Blackadar}
\newblock Operator algebras. Theory of $C^\ast$-algebras and von Neumann algebras,
\newblock {\em Berlin:Springer}, 2006.

\bibitem{CD}
M.~Capitaine and C.~Donati-Martin.
\newblock Strong asymptotic freeness for {W}igner and {W}ishart matrices.
\newblock {\em Indiana Univ. Math. J.}, 56(2):767--803, 2007.

\bibitem{ching73}
W.-M.~Ching.
\newblock Free products of von Neumann algebras.
\newblock {\em Trans. Amer. Math. Soc.} 178 (1973) 147-163.

\bibitem{Coh06}
P.~M. {Cohn}.
\newblock Free ideal rings and localization in general rings.
\newblock{\em Cambridge university press}, 3, 2006.

\bibitem{CR94}
P.~M. {Cohn} and C.~{Reutenauer}.
\newblock {A normal form in free fields}.
\newblock {\em {Can. J. Math.}}, 46(3):517--531, 1994.

\bibitem{CR99}
P.~M. {Cohn} and C.~{Reutenauer}.
\newblock {On the construction of the free field}.
\newblock {\em {Int. J. Algebra Comput.}}, 9(3-4):307--323, 1999.

\bibitem{C94}
A.~{Connes}.
\newblock {\em {Noncommutative geometry. Transl. from the French by Sterling
  Berberian}}.
\newblock San Diego, CA: Academic Press, 1994.

\bibitem{CS05}
A.~{Connes} and D.~{Shlyakhtenko}.
\newblock {\(L^2\)-homology for von Neumann algebras}.
\newblock {\em {J. Reine Angew. Math.}}, 586:125--168, 2005.

\bibitem{Con90}
J. B.~{Conway},
\newblock A course in functional analysis
\newblock {\em New York etc.: Springer--Verlag}, 1990.

\bibitem{MR3205602}
B.~{Collins} and C.~{Male}.
\newblock The strong asymptotic freeness of {H}aar and deterministic matrices.
\newblock {\em Ann. Sci. \'Ec. Norm. Sup\'er. (4)}, 47(1):147--163, 2014.

\bibitem{p1}
B.~{Collins}, A.~Guionnet and F.~Parraud.
\newblock On the operator norm of non-commutative polynomials in deterministic matrices and iid GUE matrices. \newblock {\em Cambridge Journal of Mathematics}, 10(1):195--260, 2022.

\bibitem{DR97}
G.~{Duchamp} and C.~{Reutenauer}.
\newblock Un crit\`ere de rationalit\'e provenant de la g\'eom\'etrie non commutative.
\newblock {\em Invent. Math.} 128(3):613--622, 1997.

\bibitem{EKN20}
L. {Erd\"os}, T. {Kr\"uger}, and Y. {Nemish}.
\newblock Scattering in quantum dots via noncommutative rational functions.
\newblock Ann. Henri Poincar\'e 22, No. 12, 4205-4269 (2021).

\bibitem{FQ83}
M. Fannes and J. Quaegebeur. 
\newblock Central limits of product mappings between CAR algebras,
\newblock {\em Publ. RIMS, Kyoto Univ.} 19 (1983), 469-491.

\bibitem{HS07}
U. Haagerup and H. Schultz. 
\newblock{Brown measures of unbounded operators affiliated with a finite von Neumann algebra.}
\newblock{\em Math. Scand.} 100 (2007), no. 2, 209--263.

\bibitem{HT}
U.~Haagerup and S.~Thorbj{\o}rnsen.
\newblock A new application of random matrices: {${\rm Ext}(C^*_{\rm red}(\mathbb F_2))$} is not a group.
\newblock {\em Ann. of Math.}, 162(2):711--775, 2005.

\bibitem{HMS18}
J.~W. {Helton}, T.~{Mai}, and R.~{Speicher}.
\newblock {Applications of realizations (aka linearizations) to free
  probability}.
\newblock {\em {J. Funct. Anal.}}, 274(1):1--79, 2018.

\bibitem{HW15}
P. {Hrube\v{s}} and A. {Wigderson}.
\newblock Non-commutative arithmetic circuits with division.
\newblock {\em {Theory Comput.}}, 11, 357--393, 2015.

\bibitem{HWP80}
R. L. Hudson, M. D. Wilkinson, and S. N. Peck, 
\newblock Translation-invariant integrals, and Fourier Analysis on Clifford and Grassmann algebras, 
\newblock {\em J. Functional Analysis}, 27(1980), 68-87.

\bibitem{K-VV09}
D.~S. {Kaliuzhnyi-Verbovetskyi} and V.~{Vinnikov}.
\newblock {Singularities of rational functions and minimal factorizations: the
  noncommutative and the commutative setting}.
\newblock {\em {Linear Algebra Appl.}}, 430(4):869--889, 2009.

\bibitem{K-VV12}
D.~S. {Kaliuzhnyi-Verbovetskyi} and V.~{Vinnikov}.
\newblock {Noncommutative rational functions, their difference-differential
  calculus and realizations}.
\newblock {\em {Multidimensional Syst. Signal Process.}}, 23(1-2):49--77, 2012.

\bibitem{Linnell1993}
P. A.~{Linnell}. 
\newblock Division rings and group von Neumann algebras. 
\newblock {\em Forum Math.}, 5(6):561--576, 1993.

\bibitem{Linnell2000}
P. A.~{Linnell}. 
\newblock A rationality criterion for unbounded operators. 
\newblock {\em {J. Funct. Anal.}}, 171(1):115--121, 2000.

\bibitem{MY21}
Z.~{Ma} and F.~{Yang}.
\newblock Sample canonical correlation coefficients of high-dimensional random vectors with finite rank correlations
\newblock {\em arXiv preprint arXiv:2102.03297}, 2021. 

\bibitem{male}
C. Male.
\newblock The norm of polynomials in large random and deterministic matrices.
With an appendix by Dimitri Shlyakhtenko.	
{\em Probab. Theory Related Fields} 154, no. 3-4, 477-532, 2012.

\bibitem{MSY18}
T.~{Mai}, R.~{Speicher}, and S.~{Yin}.
\newblock The free field: zero divisors, Atiyah property and realizations via unbounded operators.
\newblock {\em arXiv preprint arXiv:1805.04150v2}, 2018. 

\bibitem{MSY19}
T.~{Mai}, R.~{Speicher}, and S.~{Yin}.
\newblock The free field: realization via unbounded operators and Atiyah property.
\newblock {\em arXiv preprint arXiv:1905.08187}, 2019. 

\bibitem{MR3585560}
J.~A. {Mingo} and R.~{Speicher}.
\newblock {\em Free probability and random matrices}, volume~35 of {\em Fields
  Institute Monographs}.
\newblock Springer, New York; Fields Institute for Research in Mathematical  Sciences, Toronto, ON, 2017.

\bibitem{Miyagawa-MSc}
A.~{Miyagawa}.
\newblock The estimation of non-commutative derivatives and the asymptotics for the free field in free probability theory.
\newblock {\em MSc Thesis, Kyoto University}, 100, 2021.

\bibitem{MN36}
F. J.~{Murray} and J.~ {von Neumann},
\newblock On Rings of Operators,
\newblock {\em Ann. Math.}, 37(1):116--229, 1936.

\bibitem{p2}
F.~{Parraud}.
\newblock On the operator norm of non-commutative polynomials in deterministic matrices and iid Haar unitary matrices. 
\newblock {\em Probability Theory and Related Fields}, 182(3), 751-806, 2022.

\bibitem{SCH}
H.~Schultz.
\newblock Non-commutative polynomials of independent Gaussian random matrices. {T}he real and symplectic cases.
\newblock {\em Probab. Theory Related Fields}, 131(2):261--309, 2005.

\bibitem{MFO2019}
R.~{Speicher}.
\newblock Regularity of non-commutative distributions and random matrices.
\newblock {\em Publications of MFO}, OWR-2019-56, p3513.

\bibitem{Ter81}
M.~{Terp},
\newblock $L_p$ spaces associated with von Neumann algebras,
\newblock {\em Math. Institute, Copenhagen Univ.}, 1981.

\bibitem{Var18}
C.~{Vargas}.
\newblock {A general solution to (free) deterministic equivalents}.
\newblock In {\em {Contributions of Mexican mathematicians abroad in pure and
  applied mathematics. Second meeting ``Matem\'aticos Mexicanos en el Mundo'',
  Centro de Investigaci\'on en Matem\'aticas, Guanajuato, Mexico, December
  15--19, 2014}}, pages 131--158. Providence, RI: American Mathematical Society
  (AMS); M\'exico: Sociedad Matem\'atica Mexicana, 2018.

\bibitem{Voi85}
D.~{Voiculescu}.
\newblock Symmetries of some reduced free product $C^*$-algebras. 
\newblock In: Araki, H., Moore, C.C., Stratila, \c{S}V., Voiculescu, DV. (eds) Operator Algebras and their Connections with Topology and Ergodic Theory. Lecture Notes in Mathematics, vol 1132. Springer, Berlin, Heidelberg.

\bibitem{MR1094052}
D.~{Voiculescu}.
\newblock Limit laws for random matrices and free products.
\newblock {\em Invent. Math.}, 104(1):201--220, 1991.

\bibitem{Voi94}
D.~{Voiculescu}.
\newblock {The analogues of entropy and of Fisher's information measure in free
  probability theory. II}.
\newblock {\em {Invent. Math.}}, 118(3):411--440, 1994.

\bibitem{MR1601878}
D.~{Voiculescu}.
\newblock A strengthened asymptotic freeness result for random matrices with
  applications to free entropy.
\newblock {\em Internat. Math. Res. Notices}, (1):41--63, 1998.

\bibitem{xx9}
J.~{Vol\v{c}i\v{c}}.
\newblock Matrix coefficient realization theory of noncommutative rational functions.
\newblock {\em Journal of Algebra}, 499:397--437, 04 2018.

\bibitem{Vol2021}
J.~{Vol\v{c}i\v{c}}.
\newblock Hilbert's 17th problem in free skew fields.
\newblock Forum Math. Sigma 9, Paper No. e61, 21 p. (2021).

\bibitem{Yin18}
S.~{Yin}.
\newblock {Non--commutative rational functions in strongly convergent random variables}.
\newblock {\em {Adv. Oper. Theory}}, 3(1):178--192, 2018.

\bibitem{Yin20}
S.~{Yin}.
\newblock On the rational functions in non-commutative random variables.
\newblock PhD thesis Universit\"at des Saarlandes, 2020.

\end{thebibliography}

\end{document}